\documentclass[11pt,Letterpaper, usenames,dvipsnames,reqno]{amsart}

\usepackage{amsmath, amscd, amsthm, amssymb}

\usepackage{hyperref}

\usepackage{epic}
\usepackage{amscd}
\usepackage{amsfonts}
\usepackage{comment}

\usepackage[margin=1in]{geometry}

\usepackage[all,arc]{xy}
\usepackage{enumerate}
\usepackage{mathrsfs}
\usepackage{color}   
\usepackage{hyperref}
\hypersetup{
    colorlinks=true, 
    linktoc=all,     
    linkcolor=blue,  
}
\usepackage{amsthm}
\usepackage{amsmath}
\usepackage{graphicx}
\usepackage{wasysym}
\usepackage{pgf,tikz}
\usetikzlibrary{positioning}
\usepackage{ marvosym }
\usepackage{tikz-cd}
\usepackage{ textcomp }
\usepackage{bbm}
\usepackage{ stmaryrd }

\def\C{{\mathbf C}}
\def\R{{\mathbf R}}

\def\Z{{\mathbf Z}}
\def\Q{{\mathbf Q}}
\def\A{{\mathbf A}}

\newtheorem{theorem}{Theorem}[subsection]

\newtheorem{lemma}[theorem]{Lemma}

\newtheorem{proposition}[theorem]{Proposition}

\newtheorem{corollary}[theorem]{Corollary}

\newtheorem{claim}[theorem]{Claim}

\theoremstyle{definition}

\theoremstyle{remark}
\newtheorem{remark}[theorem]{Remark}

\newcommand{\mm}[4]{\left(\begin{smallmatrix} #1 & #2\\ #3 & #4\end{smallmatrix}\right)}

\DeclareMathOperator{\tr}{tr}

\DeclareMathOperator{\SO}{SO}
\DeclareMathOperator{\Spin}{Spin}

\DeclareMathOperator{\Sp}{Sp}

\DeclareMathOperator{\SU}{SU}

\DeclareMathOperator{\SL}{SL}
\DeclareMathOperator{\GL}{GL}

\DeclareMathOperator{\diag}{diag}

\def\g{{\mathfrak g}}

\def\h{{\mathfrak h}}

\def\m{{\mathfrak m}}

\def\sl{{\mathfrak {sl}}}

\def\su{{\mathfrak {su}}}

\begin{document}
\title{Exceptional Siegel Weil theorems for compact $\Spin_8$}
\author{Aaron Pollack}
\address{Department of Mathematics\\ The University of California San Diego\\ La Jolla, CA USA}
\email{apollack@ucsd.edu}

\thanks{AP has been supported by the Simons Foundation via Collaboration Grant number 585147,  by the NSF via grant numbers 2101888 and 2144021.}

\begin{abstract}  Let $E$ be a cubic \'etale extension of the rational numbers which is totally real, i.e., $E \otimes \mathbf{R} \simeq \mathbf{R} \times \mathbf{R} \times \mathbf{R}$.  There is an algebraic $\mathbf{Q}$-group $S_E =\Spin_{8,E}^c$ defined in terms of $E$, which is semisimple simply-connected of type $D_4$ and for which $S_E(\mathbf{R})$ is compact.  We let $G_E$ denote a certain semisimple simply-connected algebraic $\mathbf{Q}$-group of type $D_4$, defined in terms of $E$, which is split over $\mathbf{R}$.  Then $G_E \times S_E$ maps to quaternionic $E_8$.  This latter group has an automorphic minimal representation, which can be used to lift automorhpic forms on $S_E$ to automorphic forms on $G_E$.  We prove a Siegel-Weil theorem for this dual pair: I.e., we compute the lift of the trivial representation of $S_E$ to $G_E$, identifying the automorphic form on $G_E$ with a certain degenerate Eisenstein series.
	
Along the way, we prove a few more ``smaller" Siegel-Weil theorems, for dual pairs $M \times S_E$ with $M \subseteq G_E$.  The main result of this paper is used in the companion paper ``Exceptional theta functions and arithmeticity of modular forms on $G_2$" to prove that the cuspidal quaternionic modular forms on $G_2$ have an algebraic structure, defined in terms of Fourier coefficients.
\end{abstract}
\maketitle

\setcounter{tocdepth}{1}
\tableofcontents

\section{Introduction}
Suppose $G'$ is a reductive $\Q$-group with an automorphic minimal representation $V_{min}$, and $G \times S \subseteq G'$ is a dual pair.  Let $\phi \in V_{min}$ and $\Theta_\phi \in \mathcal{A}(G')$ the associated automorphic form.  Assuming the integral $\Theta_{\phi}(\mathbf{1})(g):=\int_{S(\Q)\backslash S(\A)}{\Theta(g,h)\,dh}$ converges (or can be regularized), one would like to identify explicitly the automorphic form $\Theta_{\phi}(\mathbf{1})(g) \in \mathcal{A}(G)$.  By a Siegel-Weil formula we mean an identity $\Theta_{\phi}(\mathbf{1})(g) = E_\phi(g)$, where $E_\phi(g)$ is some degenerate Eisenstein series on $G$ defined in terms of $\phi$.

Siegel-Weil theorems have a long history.  The classical setting is when $S = O(V)$ is the orthogonal group of a quadratic space of even dimension, $G = \Sp_{2m}$, $G'= \Sp_{2m \dim(V)}$ (or rather the metaplectic double cover), and $\Theta_\phi$ comes from the Weil representation.  This setting has been studied by Siegel \cite{siegelIndefiniteI}, Weil \cite{weilSW}, Rallis \cite{rallisBookOscillator}, and Kudla-Rallis \cite{kudlaRallisI,kudlaRallisII,kudlaRallisAnnals} among others.  We refer to \cite{kudlaRallisI,kudlaRallisAnnals} for a more extensive history.

The Weil representation on symplectic groups are not the only automorphic minimal representations.  The group $G'$ could, for instance, be of type $D_n$ or exceptional of type $E$.  Beautiful examples of this sort of ``exceptional" Siegel Weil theorems appear in work of Gan \cite{ganSW,ganSWSnitz,ganSWregularized} and Gan-Savin \cite{ganSavinSW}.  In this paper, we prove a family of Siegel-Weil theorems for $G'$ of type $E_8$.  We also prove some Siegel-Weil theorems when $G'$ is of type $D_6, D_7$ and $E_7$.  In all of these cases, $S$ is simply connected of type $D_4$ with $S(\R)$ compact.

Siegel-Weil theorems are frequently used in conjunction with Rankin-Selberg integrals to prove results connecting special values of $L$-functions, theta lifts, and periods of automorphic forms.  In such a framework, the Siegel-Weil Eisenstein series $E_\phi(g)$ is realized as a special value $E_\phi(g,s=s_0)$ of an Eisenstein series $E_\phi(g,s)$, and this latter Eisenstein series participates in a Rankin-Selberg integral.  Relating $E_\phi(g,s=s_0)$ to theta functions then gives a non-trivial relationship between $L$-values of the cusp forms appearing in the Rankin-Selberg integral and their theta lifts.

The situation in this paper is similar.  The main result of this paper will be used in the companion paper \cite{pollackETF} to prove that every cuspidal quaternionic modular form on $G_2$ of even weight at least $6$ lifts to an anisotropic group of type $F_4$.  Combined with the other main results of \cite{pollackETF}, this proves that the cuspidal quaternionic modular forms on $G_2$ of even weight at least $6$ have an algebraic structure, defined in terms of Fourier coefficients.  We refer to \cite{pollackETF} for more details.

We remark that our main result, Theorem \ref{thm:SWGE}, can be considered as part of the theory of $D_4$ modular forms, in the sense of \cite{weissmanD4}.  In fact, when $E = \Q\times \Q\times \Q$, Weissman hypothesized the existence of a Siegel-Weil formula as in Theorem \ref{thm:SWGE}.

\subsection{Statement of results}
We give rough version of the main result of this paper, deferring the precise statement until after all the group theoretic notation has been defined.

Let $E$ be a cubic \'etale extension of the rational numbers which is totally real, i.e., $E \otimes \R \simeq \R \times \R \times \R$.  There is an algebraic $\Q$-group $S_E =\Spin_{8,E}^c$ defined in terms of $E$, which is semisimple simply-connected of type $D_4$ and for which $S_E(\R)$ is compact.

We let $G_E$ denote a certain semisimple simply-connected algebraic $\Q$-group of type $D_4$, defined in terms of $E$, which is split over $\R$.  Then $G_E \times S_E$ maps to a group denoted $G_J$ below, which is quaternionic $E_8$.  This latter group is split at all finite places and $G_J(\R) = E_{8,4}$ has real rank four.

The group $G_J$ has an automorphic minimal representation, $V_{min}$, defined in \cite{ganMin} and studied further in \cite{ganSavin}.  We review the construction and properties of this and other automorphic minimal representations in sections \ref{sec:AMRD} and \ref{sec:AMRE} below.

Let $P_J$ be the maximal Heisenberg parabolic subgroup of $G_J$, so that $P_J$ has Levi subgroup of type $GE_{7,3}$.  Using the map $G_E \rightarrow G_J$, one can intersect $G_E \cap P_J$ to obtain a maximal Heisenberg parabolic subgroup of $G_E$, call it $P_E$.  

Let $\phi \in V_{min}$ and $\Theta_\phi$ the associated automorphic form on $G_J$.  We can realize $V_{min}$ as a submodule of $Ind_{P_J(\A)}^{G_J(\A)}(\delta_{P_J}^{5/{29}})$.  Using this realization, let $Res(\phi)$ be the restriction of $\phi$ to $G_E$.  Then $Res(\phi)$ lands in $Ind_{P_E(\A)}^{G_E(\A)}(\delta_{P_E})$.  Extending $Res(\phi)$ to a section $f_\phi(g,s)$ in $Ind_{P_E(\A)}^{G_E(\A)}(\delta_{P_E}^s)$, one can define\footnote{In the body of the paper, we normalize the parameter $s$ in these Eisenstein series differently.} an Eisenstein series $E_\phi(g,s) = \sum_{\gamma \in P_E(\Q)\backslash G_E(\Q)}{f_\phi(\gamma g,s)}$. The Eisenstein series turns out to be regular at $s=1$, and the associated automorphic form $E_\phi(g) \in \mathcal{A}(G_E)/\mathbf{1}$ (modding out by the trivial representation) is independent of the extension of $\phi$ to an inducing section $f_\phi(g,s)$.  Here is our main theorem.

Let $\Theta_\phi(\mathbf{1})(g) := \int_{S_E(\Q)\backslash S_E(\A)}{\Theta_\phi(g,h)\,dh}$ be the theta lift of $\mathbf{1}$ to $\mathcal{A}(G_E)$.  
\begin{theorem}[See section \ref{sec:Main}]\label{thm:introMain} Let the notation be as above.  Normalize the Haar measure on $S_E$ so that the automorphic quotient $S_E(\Q)\backslash S_E(\A)$ has measure $1$.  Then one has an identity of automorphic forms $\Theta_\phi(\mathbf{1})(g) = E_\phi(g)$ in $\mathcal{A}(G_E)/\mathbf{1}$.\end{theorem}

We also prove similar theorems for dual pairs of type $D_2 \times S_E \subseteq D_6$ when $E=\Q \times F$ ($F$ quadratic \'etale), $D_3 \times S_E \subseteq D_7$ when $E = \Q \times F$, and $\SL_{2,E} \times S_E \subseteq E_{7}$.  These ``smaller" theorems are, in fact, used in the proof of Theorem \ref{thm:introMain}.

\section{Group theory}\label{sec:groupTheory}
In this section, we define various groups and embeddings of groups that we use throughout the paper.

\subsection{Generalities}
We begin by recalling the following well-known result in the theory of linear algebraic groups. See \cite[Proposition 18.8, Theorem 22.53, and Theorem 23.70]{milneBook}.
\begin{proposition} Let $k$ be a field of characteristic $0$.
\begin{enumerate}
	\item Suppose $\g$ is a semisimple Lie algebra over $k$.  There exists a connected, semsisimple, simply-connected algebraic $k$-group $G(\g)$ with $Lie(G(\g)) = \g$.  The group $G(\g)$ with these properties is unique up to isomorphism.
	\item Suppose $H$ is an algebraic $k$-group and $L: \g \rightarrow Lie(H)$ a morphism of Lie algebras over $k$.  Then there exits a unique map of algebraic groups $G(\g) \rightarrow H$ whose differential is $L$.
\end{enumerate}
\end{proposition}

We also recall:
\begin{proposition} Let $k$ be a field of characteristic $0$.  Suppose $G,S$ are connected semisimple algebraic $k$-groups, and each maps to an algebraic $k$-group $H$ via maps $\iota_G, \iota_S$.  Suppose moreover that the differential $d\iota_{G}: Lie(G) \rightarrow Lie(H)$ lands in $Lie(H)^{S}$.  Then $\iota_G(G)$ lands in the centralizer of $\iota_S(S)$, so that one obtains a map $\iota_G \times \iota_S: G \times S \rightarrow H$.
\end{proposition}
\begin{proof} Let $\overline{k}$ denote the algebraic closure of $k$.  It suffices to check that if $x \in S(\overline{k})$ then $x$ centralizes $\iota_G(G)$.  For this, consider the map $\iota'$ given by composing $\iota_G$ with conjugation by $x$.  This gives a potentially new map $\iota': G \rightarrow H$ over $\overline{k}$.  But $d \iota' = d\iota_G$ because $d\iota_{G}: Lie(G) \rightarrow Lie(H)$ lands in $Lie(H)^{S}$.  But the differential gives a fully faithful functor $Rep(G) \rightarrow Rep(Lie(G))$ \cite[Theorem 22.53]{milneBook}, so $\iota' = \iota_G$.
\end{proof}

\subsection{The group $S_E$}\label{subsec:SE}
Let $\Theta$ denote the octonion $\Q$ algebra with positive-definite norm form $N_\Theta$.  We write $\tr_\Theta$ for the octonionic trace on $\Theta$.

Set $J = H_3(\Theta)$ the exceptional cubic norm structure.  We write a typical element $X$ of $J$ as 
\begin{equation}\label{eqn:XJ}X = \left(\begin{array}{ccc} c_1 & x_3 & x_2^* \\ x_3^* & c_2 & x_1 \\ x_2 & x_1^* & c_3 \end{array}\right).\end{equation}

Let $E$ be an \'etale cubic extension of $\Q$ that is totally real.  We assume given an embedding $E \hookrightarrow J$ with the following properties:
\begin{enumerate}
	\item $E$ lands in $H_3(\Q) \subseteq J$;
	\item if $N_J$ denotes the cubic norm on $J$ and $N_E$ the cubic norm on $E$, then $N_J(x) = N_E(x)$ for all $x \in E$;
	\item $1 \in E$ maps to $1 \in J$;
	\item if $E = \Q \times \Q \times \Q$, then the map $E \rightarrow J$ is $(c_1,c_2,c_3) \mapsto \diag(c_1,c_2,c_3)$;
	\item if $E = \Q \times F$ with $F$ a real quadratic field, then $(1,0) \mapsto \diag(1,0,0)$, $(0,1) \mapsto \diag(0,1,1)$, and in general the image of $X \in E$ has $x_2$ and $x_3$ coordinates equal to $0$.
\end{enumerate}

Let $V_E$ denote the orthogonal complement of $E \subseteq J$ under the trace pairing on $J$.  Let $M_J^1$ denote group of linear automorphisms of $J$ that fix the cubic norm $N_J$.  Then $M_J^1$ is simply connected of type $E_6$.  We let $S_E$ denote the subgroup of $M_J^1$ that is the identity on $E$.  The group $S_E$ is connected, simply connected of type $D_4$ and has $S_E(\R)$ compact.  

We set $E_{sp} = \Q \times \Q \times \Q$.  The group $S_{E_{sp}}$ can be recognized as the group of triples $(g_1,g_2, g_3) \in \SO(\Theta)$ with $(g_1 x_1, g_2 x_2, g_3 x_3)_{\tr_\Theta} = (x_1, x_2, x_3)_{\tr_\Theta}$ for all $x_1, x_2, x_3 \in \Theta$, where $(x_1, x_2, x_3)_{\tr_\Theta} = \tr_\Theta(x_1(x_2x_3))$.  We identify $H_2(\Theta) \subseteq J$ as those elements with $c_1, x_2,$ and $x_3$ coordinates equal to $0$.  When $E = \Q\times F$, then $V_E = \Theta^2 \oplus \Theta_F$, where $\Theta_F$ is the orthogonal complement in $H_2(\Theta)$ to the image of $F$.

\subsection{Exceptional groups}\label{subsec:excgps}
For a cubic norm structure $A$ over a field $k$ of characteristic $0$, we have the Freudenthal space $W_A = k \oplus A \oplus A^\vee \oplus k$.  This space comes equipped with a symplectic form and a quartic form, see \cite[Section 2.2]{pollackQDS}.  We let $H_J$ denote the identity component of the group of linear automorphisms of that preserve these forms up to similitude $\nu: H_J \rightarrow \GL_1$ (see \cite[Section 2.2]{pollackQDS}) and $H_J^1$ the kernel of $\nu$.  When $A = J$ is the exceptional cubic norm structure, $H_J^1$ is simply connected of type $E_7$; it is split at every finite place and $H_J^1(\R) = E_{7,3}$.  When $A = E$ is cubic \'etale over $k$, then $H_E^1 = \SL_{2,E}$.  (For an explicit map $\SL_{2,E} \rightarrow H_E^1$, one can see \cite[Section 4.4]{pollackLL}.)  We have a map $H_E^1 \times S_{E} \rightarrow H_J^1$.

We recall from \cite[Section 4]{pollackQDS} the Lie algebra $\g(A)$.  One has $\g(A) = \sl_2 \oplus \h(A)^0 \oplus (V_2 \otimes W_A)$.  Here $\h(A)^0 = Lie(H_A^1)$ and $V_2$ denotes the standard representation of $\sl_2$.  We denote by\footnote{This notation slightly conflicts with the notation of \cite{pollackQDS}, where $G_A$ denoted the adjoint group associated to this Lie algebra.} $G_A$ the connected, simply connected group with $Lie(G_A) = \g(A)$.  When $A =J$, $G_J$ is of type $E_8$, split at every finite place and $G_J(\R) = E_{8,4}$.  When $A = E$, $G_E$ is of type $D_4$ with $G_E(\R)$ split.

The group $G_J$ has rational root system of type $F_4$.  When $E = \Q \times \Q \times \Q$, $D_E$ is split and has rational root system of type $D_4$; when $E = \Q \times F$ with $F$ field, $D_E$ has rational root system of type $B_3$, and when $E$ is a field, $D_E$ has rational root system of type $G_2$.  The group $M_J^1$ has rational root system of type $A_2$, and $H_J^1$ has rational root system of type $C_3$.

\subsection{Classical root types}\label{subsec:classicalgps}
We now define some groups that we will use that have classical root types.

Let $H = \Q \oplus \Q$ be a hyperbolic plane.  Set $V_{10} = H \oplus \Theta$ with bilinear form $((h_1,\theta_1),(h_2,\theta_2)) = (h_1,h_2) - (\theta_1,\theta_2)_{\Theta}$.  Let $V_{12} = H \oplus V_{10} = H^2 \oplus \Theta$ and $V_{14} = H \oplus V_{12} = H^3 \oplus \Theta$.  We let $G_{5,\Theta}$ be the simply-connected cover of $\SO(V_{10})$, and similarly define $G_{6,\Theta}$, $G_{7,\Theta}$ as the simply connected cover of $\SO(V_{12})$, respectively, $\SO(V_{14})$.

For a quadratic \'etale extension $F$ of $\Q$, we let $V_{4,F} = H \oplus F$ with quadratic form $q(h,f)= q_H(h) + N_F(f)$.  Observe that one has $V_{10,\Theta} = F \oplus \Theta_F$, so that $V_{12,\Theta} = V_{4,F} \oplus \Theta_F$.  Similarly, we set $V_{6,F} = H^2 \oplus F = H \oplus V_{4,F}$, and $V_{14,\Theta} = V_{6,F} \oplus \Theta_F$.  Let $G_{2,F}$ be the simply connected cover of $\SO(V_{4,F})$ and let $G_{3,F}$ be the simply connected cover of $\SO(V_{6,F})$.  From the inclusions $\SO(V_{4,F}) \times \SO(\Theta_F) \subseteq \SO(V_{12,\Theta})$ and  $\SO(V_{6,F}) \times \SO(\Theta_F) \subseteq \SO(V_{14,\Theta})$, we obtain maps $G_{2,F} \times S_{\Q \times F} \rightarrow G_{6,\Theta}$ and $G_{3,F} \times S_{\Q \times F} \rightarrow G_{7,\Theta}$.

\subsection{Parabolic subgroups}\label{subsec:parabolic}
We will need to keep track of numerous parabolic subgroups of various reductive groups.  We use the following naming convention: Suppose $G$ is a reductive group, with a fixed split torus $T$ and simple roots $\alpha_1,\ldots, \alpha_r$ for a rational root system $\Phi(G,T)$.  Then for each subset $I$ of the set of simple roots $\{\alpha_1,\ldots,\alpha_r\}$, there is an associated standard parabolic $P_{G,I}=M_{G,I} N_{G,I}$ with those simple roots appearing in its unipotent radical $N_{G,I}$.  Thus if $I$ is a singleton then $P_{G,I}$ is maximal, whereas if $I = \{\alpha_1,\ldots,\alpha_r\}$ then $P_{G,I} =: P_{G,0}$ is minimal.  We will also write $P_{G,j}$ for $P_{G, \{\alpha_j\}}$, $P_{G,jk}$ for $P_{G,\{\alpha_j,\alpha_k\}}$ and so on.

\section{General facts about automorphic forms and representation theory}\label{sec:general}
In this section, we collect various general facts about automorphic forms and representation theory that we will use later in the paper.

\subsection{Weyl groups}
Suppose $G$ is a reductive group with a split maximal torus $T \subseteq G$.  Fix a set of simple roots $\Delta= \{\alpha_1,\ldots,\alpha_r\}$ for the roots $\Phi(G,T)$; $\Phi(G,T)^+$ denotes the set of positive roots.  Let $W$ denote the Weyl group of $\Phi(G,T)$.  If $P = MN$ is a standard parabolic, we write $W_M$ for the subgroup of $W$ generated by the simple reflections corresponding to elements $\Delta \cap \Phi(M,T)$.  

Following \cite[Section 1]{casselmanBook}, set
\[[W/W_M] = \{w \in W: w(\alpha) \in \Phi(G,T)^+ \forall \alpha \in \Delta \cap \Phi(M,T)\},\]
\[[W_M\backslash W] = \{w \in W: w^{-1}(\alpha) \in \Phi(G,T)^+ \forall \alpha \in \Delta \cap \Phi(M,T)\}.\]
If $P= MN$, $Q = LV$ are two standard parabolic subgroups, then we set
\[[W_L \backslash W/W_M] = [W_L \backslash W] \cap [W/W_M].\]

\subsection{Intertwining operators}
We now review standard facts about intertwining operators.  Suppose $P_0$ is the minimal standard parabolic for the root system $\Phi(G,T)$.  Let $k$ be a local field, either $p$-adic or archimedean.  Let $\chi: P_0(k) \rightarrow \C^\times$ be a character.  Consider the induced representation $Ind_{P_0(k)}^{G(k)}(\delta_{P_0}^{1/2} \chi)$.  If $w \in W$, we will defined the intertwining operator 
\[M(w): Ind_{P_0(k)}^{G(k)}(\delta_{P_0}^{1/2} \chi) \rightarrow Ind_{P_0(k)}^{G(k)}(\delta_{P_0}^{1/2} w(\chi))\]
associated to $w$.

Let $U_\alpha$ be the root subgroup corresponding to the root $\alpha \in \Phi(G,T)$ and set $N_w = \prod_{\alpha}{U_\alpha}$ with the product over positive roots $\alpha$ for which $w^{-1}(\alpha)$ is negative.  If $w \in W$ and $f \in Ind_{P_0(k)}^{G(k)}(\delta_{P_0}^{1/2} \chi)$, we set
\[M(w)f(g) = \int_{N_w(k)}{f(w^{-1}ng)\,dn}\]
if the integral is absolutely convergent.  In this case, $M(w):  Ind_{P_0(k)}^{G(k)}(\delta_{P_0}^{1/2} \chi) \rightarrow Ind_{P_0(k)}^{G(k)}(\delta_{P_0}^{1/2} w(\chi))$ is $G(k)$-intertwining.  If $w = w_1 w_2$ with $\ell(w) = \ell(w_1) + \ell(w_2)$, and each of $M(w_1)$, $M(w_2)$ are defined (i.e., defined by absolutely convergent integrals), then so is $M(w)$, and $M(w) = M(w_1) \circ M(w_2)$.

Let $h_\alpha: \GL_1 \rightarrow T$ be the cocharacter associated to the root $\alpha$.  If $w = s_\alpha$ is a simple reflection, and $N_w = U_\alpha$ is one-dimensional, then the integral defining $M(s_\alpha)$ is absolutely convergent if $\chi(h_\alpha(t)) = |t|^s$ with $Re(s) > 0$.

We will frequently be in the following situation: $P \subseteq G$ is a maximal parabolic subgroup, $\nu: P \rightarrow \GL_1$ is a character, and $I(s):= Ind_{P(k)}^{G(k)}(|\nu|^{s})$. We assume $\delta_P= |\nu|^{s_P}$ with $s_P >0$ as opposed to $s_P < 0$. We set $\lambda_s = \delta_{P_0}^{-1/2} |\nu|^s$, so that $I(s) \subseteq Ind_{P_0(k)}^{G_0(k)}(\delta_{P_0}^{1/2} \lambda_s)$.  Now suppose $w \in [W/W_M]$ where $P = MN$.  Then if $\alpha > 0$ and $w^{-1}(\alpha) < 0$, we must have that $w^{-1}(U_\alpha) \subseteq \overline{N}$, the unipotent radical opposite to $P$.   In this case, $M(w)$ is absolutely convergent for $Re(s) >> 0$, and has a meromorphic continuation in $s$, in the following sense.  Given $s_0$, there is an integer $k$ so that if $f(g,s) \in I(s)$ is a flat section, then $(s-s_0)^k M(w)f(g,s)$ is meromorphic in $s$ and regular at $s_0$.  In this case, \[(s-s_0)^k M(w): I(s_0)\rightarrow Ind_{P_0(k)}^{G_0(k)}(\delta_{P_0}^{1/2} w(\lambda_s))\]
is $G(k)$-intertwining.

In the situation of the above paragraph, there is one intertwiner to which we give a fixed name and notation: the long intertwining operator.  Namely, the set $[W/W_M]$ has a unique longest element (see \cite[Proposion 1.1.4]{casselmanBook}), which we denote by $w_0$.  While this notation is independent of the group $G$, the underlying group $G$ will always be clear from context.  When we say the long intertwining operator, we mean $M(w_0)$.

Throughout, we normalize the Haar measure on a reductive group $G$ over $\Q$ so that the unipotent quotient groups $U_\alpha(\Q)\backslash U_\alpha(\A)$ have measure $1$.

\subsection{Eisenstein series}\label{subsec:Eis}
Suppose $G$ is a reductive group over $\Q$, and $P \subseteq G$ is a maximal parabolic subgroup.  Let $\nu: P \rightarrow \GL_1$ be a character, and $I(s) = Ind_{P}^{G}(|\nu|^s)$ (sometimes meaning a local induction, and sometimes meaning a global induction).  We will always write $I_v(s)$ for the local induction at a place $v$ and $I_f(s)$ for the induction at the finite places,   We let $s_P$ be the real number for which $\delta_P = |\nu|^{s_P}$, and we assume $s_P$ is positive (as oppposed to negative).  Let $f(g,s) \in I(s)$ be a global flat section.  The Eisenstein series associated to $f$ is
\[E(g,f,s) = \sum_{\gamma \in P(\Q)\backslash G(\Q)}{f(\gamma g,s)}.\]
The sum converges absolutely for $Re(s) > s_P$ and defines an analytic function of $s$ in that range.  The Eisenstein series has meromorphic continuation in $s$.  

Suppose $f_\infty(g,s)$ is a $K_\infty$-finite flat section, where $K_\infty \subseteq G(\R)$ is a flat section.  Fixing $f_\infty$, we can let $f_{fte} \in I_f(s)$ vary, and consider the Eisenstein map $Eis: I_f(s) \rightarrow \mathcal{A}(G)$ for $Re(s) > s_P$.  Now, given $s_0$, there is an integer $k$ so that $(s-s_0)^k E(g,f,s)$ is regular at $s=s_0$ for all $f_{fte}(g,s) \in I_f(s)$, and in that case, the Eisenstein map is well-defined and intertwining from $I_{f}(s_0)$ to $\mathcal{A}(G)$.

Recall that to test if the Eisenstein series $(s-s_0)^kE(g,f,s)$ is regular at $s_0$, it suffices to see that its constant term along any parabolic $Q$ of $G$ is regular.  If $Q = LV$ is standard, the constant term $E_V(g,s)$ is a finite sum of Eisenstein series for the Levi $L$, one for each element of $[W_L\backslash W/W_M]$, where $P = MN$.  Precisely, one has the following relation.

Suppose $w \in [W_L\backslash W/W_M]$ and $f(g,s) \in I(s)$, global induction.  Then one has the intertwined inducing section $f^w(g,s):=M(w)f(g,s)$, and the new Eisenstein series 
\[E^{w}(g,f,s) = \sum_{\gamma (wPw^{-1} \cap L)(\Q)\backslash L(\Q)}{f^w(\gamma g,s)}.\]
Then 
\[E_V(g,s) = \sum_{w \in [W_L\backslash W/W_M]}{E^{w}(g,f,s)},\]
and everything in sight converges absolutely for $Re(s)>>0$.

The parabolic subgroup $wPw^{-1} \cap L$ of $L$ can be described in terms of simple roots.  Namely, suppose $\{\alpha_1,\ldots,\alpha_r\} = I \sqcup J$ are the simple roots for $\Phi(G,T)$ with those elements of $I$ appearing in the unipotent radical $V$ of $Q$ and those elements of $J$ appearing in $\Phi(L,T)$.  To describe $wPw^{-1} \cap L$, we would like to know those simple roots $\beta$ of $\Phi(L,T)$ which appear in the unipotent radical of $wPw^{-1} \cap L$.  Equivalently, we would like to find those roots $\beta \in J$ for which $U_{-\beta}$ is not contained in $wPw^{-1} \cap L$.  We therefore set 
\[\Delta^w(L) = \{\beta \in J=\Phi(L,T): w^{-1}(\beta) \in \Phi(N,T)\}.\]
The parabolic subgroup $wPw^{-1} \cap L$ of $L$ is thus $P_{L,\Delta^w(L)}$, in the notation of subsection \ref{subsec:parabolic}.  We call the set $\Delta^w(L)$ \emph{the associated simple roots of $w$}.  Note that $\Delta^w(L)$ is sometimes empty, in which case $E^w(g,f,s) = f^w(g,s)$ is just a single term.

\subsection{Cuspidal and Eisenstein projection}
Suppose $G$ is a reductive group, either actually semisimple, or we fix some unitary central character.  Then if $\varphi \in \mathcal{A}(G)$ is a (moderate growth) automorphic form, we can consider the linear functional on the space of cusp forms $\xi \mapsto \int_{[G]}{\xi(g) \overline{\varphi(g)}\,dg}$.  Because the space of cusp forms is a Hilbert space, this linear functional is represented by a cusp form, called the cuspidal projection of $\varphi$, $\varphi_{cusp}(g)$.  One checks immediately that $\varphi \mapsto \varphi_{cusp}$ is linear and $G(\A_f)$-equivariant.

We define $\varphi_{Eis} = \varphi - \varphi_{cusp}$.  Note that $\varphi_{Eis}$ is orthogonal to the space of cusp forms, and the map $\varphi \mapsto \varphi_{Eis}$ is $G(\A_f)$-equivariant.

\subsection{Irreducibility and unitarizability of principal series}
We will require the following well-known theorem.

\begin{theorem}\label{thm:irredUnit} Let $G$ be a split semisimple group over a $p$-adic field $k$.  Let $P \subseteq G$ be a maximal parabolic subgroup.  If $Re(s) > 1$, then the representation $Ind_P^G(\delta_P^{s})$ is irreducible and not unitarizable.
\end{theorem}
\begin{proof} The irreducibility can be proved using the criterion in \cite[Theorem 3.1.2]{jantzenMAMS}.  
	
Now, the set of $s$ for which $Ind_P^G(\delta_P^s)$ has a unitarizable subquotient is compact.  This is a theorem of Tadic, see \cite[Lemma 5.1]{muicG2}. If for any $s_0$ with $Re(s_0) > 1$ one had that $Ind_P^G(\delta_P^{s_0})$ were unitarizable, then by continuity all $Ind_P^G(\delta_P^{s})$ with $Re(s)>1$ would be unitarizable, violating the compactness.  This proves the theorem.
\end{proof}

\subsection{Moving between isogenous groups}\label{subsec:isog}
Throughout the paper, we will prove Siegel-Weil identities and theorems about Eisenstein series on groups that are semisimple and simply connected.  We will then utilize these results on isogenous groups.  We explain now the principle that lets us go between isogenous groups in the context of these sorts of results.

Suppose $G$ is reductive, and let $G_{sc}$ the simply-connected cover of the derived group of $G$.  There is a canonical map $\pi_{sc}:G_{sc} \rightarrow G$. Let $P$ be a parabolic subgroup of $G$. We begin with the following well-known lemma.

\begin{lemma}\label{lem:parsingle}For any field $k$ of characteristic $0$, $P(k) \backslash G(k)/ G_{sc}(k)$ is a singleton.
\end{lemma}
\begin{proof}
	The unipotent elements of $G(k)$ are in the image of $G_{sc}(k)$.  Consequently, representatives of the Weyl group of $G$ are in the image of $\pi_{sc}$ as well.  Combining the last two statements, the lemma follows from the Bruhat decomposition.
\end{proof}

Suppose now $G$ is defined over $\Q$, $K_\infty \subseteq G(\R)$ a maximal compact subgroup, and let $U$ be a finite-dimensional representation of $K_\infty$.  Suppose we have a family of characters $\chi_s: P(\A) \rightarrow \C^\times$ depending on a complex parameter $s$.  Suppose also that $f_\infty(g,s) \in I_{P(\R)}^{G(\R)}(\chi_s) \otimes U$ is $K_\infty$-equivariant, i.e., $f_\infty(gk,s) = k^{-1} f_\infty(g,s)$ for all $g \in G(\R)$ and $k \in K_\infty$.  Let $f_{fte}(g,s) \in Ind_{P(\A_f)}^{G(\A_f)}(\chi_s)$ be some inducing section, $f(g,s) = f_{fte}(g,s) f_{\infty}(g,s)$, and $E(g,f,s)$ the Eisenstein series, defined by an absolutely convergent summation for $Re(s) >> 0$.  We assume $f(g,s)$ is flat.

\begin{lemma}\label{lem:isogEis} Let $Z$ denote the center of $G$.  Assume that $G(\R) = Z(\R)G_{sc}(\R) K_\infty$. Let the other notation and assumptions be as above.  Then the restriction of $E(g,f,s)$ to $G_{sc}(\A)$ is regular (or has a pole of order at most $k$) at some special piont $s=s_0$ for all $f_{fte}(g,s)$ if and only if $E(g,f,s)$ is regular (or has a pole of order at most $k$) on $G(\A)$ for all $f_{fte}(g,s)$.
\end{lemma}
\begin{proof} Let $k$ be the highest order of pole for the Eisenstein series on $G$ and $k_{sc}$ the highest order of pole for the Eisenstein series on $G_{sc}$. By Lemma \ref{lem:parsingle},$P \cap G_{sc} =: P_{sc}$ is a parabolic subgroup of $G_{sc}$, and the restriction map from the induced representation on $G$ to the induced representation on $G_{sc}$ is surjective.  Consequently, $k_{sc} \leq k$.  For the reverse inequality, we have the $(s-s_0)^kE(g,f,s)$ gives an intertwining map from $I_{fte}(\chi_s)$ to $\mathcal{A}(G)$.  There exists $f_{fte}(g,s)$ so that the associated Eisenstein series value is non-zero.  By the fact that the map is intertwining, we can assume that $(s-s_0)^kE(g_\infty,f,s)$ is nonzero at $s=s_0$.  But now because $G(\R) = Z(\R)G_{sc}(\R) K_\infty$, this Eisenstein series is nonzero on $G_{sc}(\R)$.  Hence $k_{sc} \geq k$.
\end{proof}

We now explain how we go between Siegel-Weil formulas proved for simply-connected groups, and Siegel-Weil formulas for more general groups.  So, suppose we have a commuting pair $G \times S \rightarrow G'$.  Let $G_1' = G'_{sc}$ and $G_{1} =G_{sc}$.  Then we have maps of Lie algebras
\[ Lie(G_1) \rightarrow Lie(G) \rightarrow Lie(G') \rightarrow Lie(G_1')\]
so there is a map $G_1 \rightarrow G_1'$.  Suppose $H$ is simply connected, so the map $H \rightarrow G'$ factors through $G_1'$.  Now $Lie(G)$ maps into $Lie(G')^{S}$, and thus $Lie(G_1')^S$, so $Lie(G_1)$ maps into $Lie(G_1')^S \subseteq Lie(G_1')$.  Consequently we have a map $G_1 \times H \rightarrow G_1'$.

Now suppose we have proved a Siegel-Weil theorem for the pair $G_1 \times S \rightarrow G_1'$, and we have a hypothetical Siegel-Weil theorem for the pair $G \times S \subseteq G'$.  That is:
\begin{enumerate}
	\item we have an induced representation on $G'(\A_f)$, $I_{G',fte}(s=s_0)$ and similarly an induced representation on $G_1'(\A_f)$;
	\item we have two linear maps $m_{theta}, m_{Eis}$, $m_{*}: I_{G',fte}(s=s_0) \rightarrow (\mathcal{A}(G) \otimes U)^{K_{G,\infty}}$;
	\item we have two linear maps $m_{theta}^{1}, m_{Eis}^{1}$, $m_{*}: I_{G_1',fte}(s=s_0) \rightarrow (\mathcal{A}(G_1) \otimes U)^{K_{G_1,\infty}}$;
	\item the maps are compatible with restrictions, i.e. $m_{*}^{1} Res_{G'}^{G_1'}(f_{fte}) = Res_{G}^{G_1} m_{*} (f_{fte})$;
	\item the maps $m_{Theta}^1 = m_{Eis}^1$ agree on $I_{G_1',fte}(s=s_0)$;
	\item the maps $m_{*}$ are $G(\A_f)$ intertwining, and the maps $m_{*}^1$ are $G_1(\A_f)$ intertwining.
\end{enumerate} 

\begin{proposition} Suppose $m_{theta}$ and $m_{Eis}$ both land in $\mathcal{A}(G)_\xi \otimes U$ for the same central character $\xi.$  Suppose also that $G(\R) = Z_{G}(\R) G_1(\R) K_{\infty,G}$.  Then in the above situation, we have a Siegel-Weil theorem for the pair $G \times S \subseteq G'$.  I.e., $m_{Theta} = m_{Eis}$ on $I_{G',fte}(s=s_0)$.
\end{proposition}
\begin{proof}  Set $d = m_{Theta}-m_{Eis}$.  Then for all $f_{fte} \in I_{G',fte}(s=s_0)$, the hypotheses imply $d(f_{fte})|_{G_1} = 0$.  Because the maps $m_{*}$ are $G(\A_f)$ intertwining, we need only verify that $d(f_{fte})|_{G(\R)} \equiv 0$ for all $f_{fte}$.  Now the proposition follows from the assumption that $G(\R) = Z_{G}(\R) G_1(\R) K_{\infty,G}$ and the fact that $d(f_{fte})|_{G_1} = 0$.
\end{proof}

\section{Automorphic minimal representations: type $D$}\label{sec:AMRD}
In this section, we consider degenerate Eisenstein series on groups of type $D_5,D_6,D_7$ that can be used to define automorphic minimal representations.  The minimal representations appear as residues of the Eisenstein series at certain special points.  Beyond reviewing these constructions, we also calculate the constant terms along maximal parabolic subgroups of the functions in the automorphic minimal representation.

\subsection{Degenerate principal series}\label{subsec:DndegPS}
We review a few facts concerning degenerate principal series on $p$-adic groups of type $D$.  Thus let $k = \Q_p$ and let $G_n$ be split, simply-connected of type $D_n$.  We set $V_{2n} = H^{n}$ over $k$.  Thus $G_n \rightarrow \SO(V_{2n})$.  Let $b_j, b_{-j}$ be the standard basis of the $j^{th}$ copy of the hyperbolic plane $H$ in $V_{2n}$.  Let $P \subseteq G$ be the parabolic stabilizing the line $k b_1$, and $\nu: P \rightarrow \GL_1$ as $p b_1 = \nu(p) b_1$.  One has $\delta_P = |\nu|^{2n-2}$.  We set $I_P^G(s) = Ind_P^G(|\nu|^s)$; the long intertwiner $M(w_0)$ relates $I(s)$ with $I(2n-2-s)$.

\begin{proposition}[Weissman \cite{weissmanFJ}]\label{prop:degPSDn} Let the notation be as above.
	\begin{enumerate}
		\item The representations $I_P^G(n)$ and $I_P^G(n-2)$ have a non-split composition series of length two.
		\item The spherical representation is a subrepresentation of $I_P^G(n-2)$ and a quotient of $I_P^G(n)$.
		\item The spherical vector generates $I_P^G(n)$.
		\item The meromorphically continued intertwining operator $M(w_0): I_P^G(s) \rightarrow I_P^G(2n-2-s)$ defines a regular intertwining map at $s=n$.  The image of the intertwining operator is the unique irreducible subrepresentation of $I_P^G(s=n-2)$; it is spherical.
	\end{enumerate}
\end{proposition}
\begin{proof}  This is essentially all contained in \cite{weissmanFJ}. The structure of a composition series of $I_P^G(s)$ is computed using the Fourier-Jacobi functor, which reduces the calculation to the case of principal series on $\SL_2$.
	
For the second part, one computes that the $c$-function for $M(w_0)$ is 
\[ c(s) = \frac{\zeta(s+1-n) \zeta(s+3-2n)}{\zeta(s) \zeta(s+2-n)}.\]
This vanishes at $s=n-2$, so the spherical representation is a subrepresentation at $s=n-2$ and a quotient at $s=n$.  The third part follows from the second part. Because the spherical vector generates $I(n)$, that $M(w_0)$ is regular follows from the fact that $c(n)$ is regular.
\end{proof}

\subsection{Type $D_5$}  Recall the group $G_5=G_{5,\Theta}$, which is simply-connected of type $D_5$. This group acts on $V_{10} = H \oplus \Theta$.  Let $b_1, b_{-1}$ be the standard basis of $H$, and $P_{G_{5}}$ the parabolic subgroup stabilizing the line $\Q b_1$.  We define $\nu$ and $I(s)$ as in the previous subsection.

Suppose $f = \otimes f_v \in  Ind(|\nu|^s)$ is a flat section for every $v$, and $f_\infty$ is spherical.   We have the Eisenstein series $E(g,f,s) = \sum_{\gamma \in P_{G_5}(\Q)\backslash G_5(\Q)}{f(\gamma g,s)}$.

\begin{proposition}\label{prop:D5reg} The Eisenstein series $E(g,f,s)$ is regular at $s=5$.\end{proposition} 
\begin{proof}  Observe that $G_5$ is split at every finite place.  We first verify the proposition for the spherical Eisenstein series, then deduce the general case from the spherical case.
	
Let $N_{P_{G_5}}$ denote the unipotent radical of $P_{G_5}$. The constant term of $E(g,f,s)$ along $N_{P_{G_5}}$ is (see subsection \ref{subsec:Eis})
\[E_{N_{P_{G_5}}}(g,f,s) = f(g,s) + M(w_0)f(g,s).\]  The $c$-function is $\frac{\zeta(s-4)\zeta(s-7)}{\zeta(s)\zeta(s-3)}$.  The archimedean interwiner is $\frac{\Gamma(s-4)}{\Gamma(s)}$ \cite[Lemma 4.1.1]{pollackNTM}.  The global zeta function $\zeta(s-4)$ has a simple pole at $s=5$, but $\zeta(s-7)$ has a simple zero at $s=5$.  The archimedean intertwiner is regular at $s=5$, so we conclude that the lemma holds in case $f$ is spherical everywhere.
	
The general case follows because the spherical vector generates the module $I_p(s)$ for $s=5$ and every $p < \infty$; see Proposition \ref{prop:degPSDn}.
\end{proof}

\subsection{Type $D_6$}\label{subsec:AMRD6}
Recall that we set $V_{12} = H^2 \oplus \Theta$, and $G_{6} = G_{6,\Theta} \rightarrow \SO(V_{12})$ the simply-connected cover.  In this subsection, we discuss the automorphic minimal representation on $G_{6}$, and compute constant terms of the automorphic forms in this representation for the standard maximal parabolic subgroups of $G_6$.

To begin, let $b_j, b_{-j}$ be the standard basis of the $j^{th}$ copy of the hyperbolic plane $H$.  We set $P_{G_6,1}$ the maximal parabolic that stabilizes the line $\Q b_1$, with associated character $\nu$.  We set $P_{G_6,2}$ the maximal parabolic that stabilizes the plane $\Q b_1 \oplus \Q b_2$.  As usual, we set $I(s) = Ind_{P_{G_6,1}}^{G_6}(|\nu|^s)$.

Set $v_j = \frac{1}{\sqrt{2}}(b_j+b_{-j})$, so that $(v_i,v_j) = \delta_{ij}$.  We let $\iota \in O(V_{12})$ be the map that exchanges $b_j$ with $b_{-j}$, and is minus the identity on $\Theta$.  We define a maximal compact subgroup of $\SO(V_{12})(\R)$ as those group elements that commute with $\iota$; similarly one has a maximal compact subgroup $K_{6,\infty}$ of $G_{6}(\R)$.  Observe that the action of $K_{6,\infty}$ on $v_1 + i v_2 \in V_{12} \otimes \C$ defines a representation $j(\bullet, i): K_{6,\infty} \rightarrow \C^\times$.

For an even integer $\ell$, let $f_{\infty,\ell}(g,s) \in I_\infty(s)$ be the flat section with $f_{\infty,\ell}(k_\infty,s) = j(k_\infty,i)^{\ell}$.  Let $f_{fte}(g,s) \in I_{f}(s)$ be a flat section, let $f_{\ell}(g,s) = f_{fte}(g,s) f_{\infty,\ell}(g,s)$ and $E(g,f_{\ell},s)$ the associated Eisenstein series.

\begin{proposition}\label{prop:D6minAut} Let the notation be as above.
\begin{enumerate}
	\item \cite[Theorem 6.4]{hanzerSavin} The Eisenstein series $E(g,f_{\ell},s)$ has at most a simple pole at $s=6$. 
	\item This simple pole is achieved if $\ell=\pm 4$ and $f_{fte}(g,s)$ is spherical.
	\item Fixing $\ell=4$ or $\ell=-4$, the residual representation is irreducible.
	\item If $\ell \in \{-2,0,2\}$, then $E(g,f_{\ell},s)$ is regular at $s=6$.
\end{enumerate}
\end{proposition}
\begin{proof} As written, part 1 is from \cite{hanzerSavin}.  Part 2 follos form the computation of the constant term of the Eisenstein series in the proof of Theorem 7.0.1 of \cite{pollackNTM}.  Part 3 again follows from \cite[Theorem 6.4]{hanzerSavin}.
	
It remains to prove the fourth part.  Because the spherical vector generates $I_{f}(s=6)$, it suffices to check the regularity on the spherical vector.   Its constant term down to the minimal parabolic $P_0$ is a sum of four terms (see Proposition 6.1.1 of \cite{pollackNTM} and the proof of Theorem 7.0.1 of that paper.). These four terms give $c$-functions (see \cite{pollackNTM} for notation)
\begin{enumerate}
	\item $c(1) = 1$
	\item $c(w_{12}) = \frac{\zeta(s-1)}{\zeta(s)} \frac{\Gamma_\C(s-1)}{\Gamma_\R(s-j)\Gamma_\R(s+j)}$
	\item $c(w_{2} w_{12}) = c(w_{12}) \frac{\zeta(s-5)\zeta(s-8)}{\zeta(s-1)\zeta(s-4)}\frac{\Gamma(s-5)}{\Gamma(s-1)}$
	\item $c(w_{12} w_{2} w_{12}) = c(w_{2} w_{12}) \frac{\zeta(s-9)}{\zeta(s-8)} \frac{\Gamma_\C(s-9)}{\Gamma_\R(s-8+v)\Gamma_\R(s-8-v)}$
\end{enumerate}
All these terms are regular at $s=6$.  Thus the spherical Eisenstein series is regular at $s=6$.  This proves the proposition.
\end{proof}

Given $f_{fte} \in I_f(s=6)$, we let $\Theta_f^+ = Res_{s=6} E(g,f_{4},s)$ and $\Theta_f^{-} = Res_{s=6}E(g,f_{-4},s)$, where we have extended $f_{fte}$ to a flat section to define the Eisenstein series map.  We set $V_{min,p}$ the unique irreducible quotient of $I_p(s=6)$, or equivalently the unique irreducible subrepresentation of $I_p(s=4)$.  We set $V_{min} = \otimes'_p V_{min,p}$.  Then the residue of the Eisenstein map gives an intertwining operator $V_{min} \rightarrow \mathcal{A}(G_{6})$ for either the $+$ or $-$ cases.

Recall the standard maximal parabolic subgroups $P_{G_6,j} = M_{G_6,j} N_{G_6,j}$ of $G_6$, for $j=1,2$.  We will now compute the constant terms $\Theta^+f(g)_{N_{G_6,j}}$; the computation for $\Theta_f^-$ is identical.

The structure of the computation is explained in section \ref{subsec:Eis}.  What one has to do is to compute which intertwining operators $M(w)$ and which Eisenstein series $E^w(g,f_4,s)$ are absolutely convergent and thus do not contribute to the residue at $s=6$.  And for those intertwining operators and Eisenstein series which can potentially contribute to the residue at $s=6$, one must make a finer computation to determine if they actually do contribute.

Let $f^1(g,s) = M(w_0)f_4(g,s)$, the result of applying the global long intertwining operator $M(w_0)$ to $f_4(g,s)$.  Then $f^1(g,s) \in I(10-s)$, and has a simple pole at $s=6$.  Indeed, one can check that there is a simple pole for the finite spherical vector, and then there is at most a simple pole for the entire induced representation $I_f(s)$ because $I_f(s=6)$ is generated by the spherical vector.  We let $\overline{f}(g) = Res_{s=6} f^1(g,s)$.
\begin{proposition}\label{prop:minD6const} Let $f(g,s) = f_{fte}(g,s) f_{\infty,4}(g,s) \in I(s)$ be a flat section and let the notation be as above.
	\begin{enumerate}
		\item One has $\Theta_f(g)_{N_{G_6,2}} = E_{\GL_2}(g,\overline{f})$, an absolutely convergent Eisenstein series of $\GL_2$-type, defined as a sum over $P_{G_6,0}\backslash P_{G_6,2}$.
		\item One has $\Theta_f(g)_{N_{G_6,1}} = \overline{f}(g)$.
	\end{enumerate}
\end{proposition}
\begin{proof}
To give notation for the computation, suppose $t \in T$ the split standard torus, and $t b_j = t_j b_j$ for $j = 1,2$.  We set $r_j(t) = |t_j|$.  Writing unramified characters of $T$ in additive notation, we have $|\nu|^s$ is $s r_1$ and $\lambda_s = \delta_{P_0}^{-1/2} |\nu|^s$ is $(s-5)r_1 + (-4)r_2$.	
	
For the first part of the proposition, observe that the set $[W_{M_{G_6,2}}\backslash W_{G_6}/W_{M_{G_6,1}}]$ has size $2$, with elements $1, w$ where $w(r_1) = -r_2$ and $w(r_2) = r_1$.  The Eisenstein series corresponding to $w=1$ is absolutely convergent at the special point $s=6$, so this term does not contribute to the residue.  The other term can be understood in terms of $f^1(g,s) = M(w_0)f_{4}(g,s)$ using Langlands functional equation for Eisenstein series.

For the second part of the proposition, one has that the set  $[W_{M_{G_6,1}}\backslash W_{G_6}/W_{M_{G_6,1}}]$ has size three.  The elements are $\{1,w_0,w_1\}$ where $w_0(r_1) = -r_1$, $w_0(r_2) = r_2$ and $w_1(r_1) = r_2$, $w_1(r_2) = r_1$.  The intertwining operator $M(w_1)$ is checked to be absolutely convergent at $s=6$.  One finds that the Eisenstein series $E^{w_1}(g,f^{w_1},s)$ is exactly the one (up to isogenous groups, see Lemma \ref{lem:isogEis}) studied in Proposition \ref{prop:D5reg}.  The result follows.
\end{proof}

\subsection{Type $D_7$}\label{subsec:AMRD7}
In this subsection, we prove results about the automorphic minimal representation on the group $G_{7} = G_{7,\Theta}$ of type $D_7$.  Recall that $V_{14} = H^3 \oplus \Theta$.  We write $b_j, b_{-j}$ for the standard basis of the $j^{th}$ copy of the hyperbolic plane $H$.

Define $\iota \in O(V_{14})$ exactly as in the case of $O(V_{12})$, so that $\iota$ exchanges $b_j$ with $b_{-j}$ and is minus the identity of $\Theta$.  We write $K_{G_7,\infty} \subseteq G_{7}(\R)$ for the associated maximal compact subgroup.  We let $v_j = \frac{1}{\sqrt{2}}(b_j + b_{-j})$ so that $(v_i,v_j) = \delta_{ij}$.  The action on $V_3:=\mathrm{Span}(v_1,v_2,v_3)$ defines a map $K_{G_7,\infty} \rightarrow O(3)$.

Suppose $\ell \geq 1$ is an integer.  The $\ell^{th}$ symmetric power $S^{\ell}(V_3)$ has an irreducible highest weight quotient of dimension $2\ell+1$; this is the usual theory of spherical harmonics.  Let $\mathbf{V}_{\ell}$ be this $2\ell+1$-dimensional space, which is an $O(3)$ and thus $K_{G_7,\infty}$ representation.

Let $f_{\infty,\ell}(g,s) \in I(s) \otimes \mathbf{V_{\ell}}$ be the flat section, for which $f_{\infty,\ell}(gk,s) = k^{-1} f_{\infty,\ell}(g,s)$ and $f_{\infty,\ell}(1,s)$ is the image of $v_1^{\ell}$ in $\mathbf{V}_{\ell}$.  Suppose $f_{fte}(g,s) \in I_{f}(g,s)$ is a flat section for this induced representation at the finite places.  Let $f_{\ell}(g,s) = f_{fte}(g,s) f_{\infty,\ell}(g,s)$ and let $E(g,f,s)$ be the associated Eisenstein series.  \textbf{We now set $\ell = 4$.}

\begin{proposition} Let the notation be as above.
\begin{enumerate}
	\item The Eisenstein series $E(g,f,s)$ has at most a simple pole at $s=7$, which is achieved by the finite-place spherical vector. 
	\item The residue defines a $G_7(\A_f)$ intertwining map $I_f(s=7) \rightarrow \mathcal{A}(G_7)$.  The image is nonzero and irreducible.
	\item The global long intertwining operator $M(w_0)$ has at worst a simple pole at $s=7$.
\end{enumerate}	
\end{proposition}
\begin{proof} That the poles of the Eisenstein series are at most simple follows from \cite[Proposition 6.3]{hanzerSavin}.  As usual, by the structure of the degenerate principal series $I(s=7)$ reviewed in Proposition \ref{prop:degPSDn}, to prove that $M(w_0)$ has at worst a simple pole, it suffices to analyze the case when the finite part of our inducing section is spherical.
	
In this case, the intertwining operator $M(w_0)$ is computed in \cite[Proposition 4.1.2]{pollackNTM} (archimedean part) and \cite[Proposition 6.2.1]{pollackNTM} (finite part). One obtains
\[c(s) = \frac{\zeta(s-6)\zeta(s-11)}{\zeta(s) \zeta(s-5)} \frac{ ((s-5)/2)_2 \Gamma(s-6) ((s-14)/2)_2}{ ((s-2)/2)_{3} \Gamma(s-2) ((s-11)/2)_{3}}.\]
Here $(z)_n = z(z+1) \cdots (z+n-1)$ is the Pochammer symbol.  One sees that $c(s)$ has a simple pole at $s=7$. This proves the third part of the proposition.

Going back to the first part, because $E(g,f_s) = E(g, M(w_0)f_s)$, and the spherical Eisenstein series is regular at $s=5$ \cite[Theorem 7.0.1]{pollackNTM}, the spherical Eisenstein series $E(g,f,s)$ has an honest pole at $s=7$.  

Part 2 of the proposition follows as in previous cases from \cite{hanzerSavin}.
\end{proof}

We write $\Theta_f(g) = Res_{s=7} E(g,f,s)$.  This is an element of $\mathcal{A}(G_{7,\Theta}) \otimes \mathbf{V}_{4}$ that is $K_{G_7,\infty}$-equivariant.

For $j = 1,2,3$, let $P_{G_7,j} = M_{G_7,j} N_{G_7,j}$ be the parabolic subgroup of $G_{7}$ that stabilizes the isotropic subspace spanned by $b_1$ through $b_j$.  We compute the constant terms of $\Theta_f$ along the $N_{G_7,j}$ for $j = 1,2,3$.

Similar to our analysis of the minimal representation on $G_6$, we set $f^1(g,s) = M(w_0)f_{4}(g,s)$ and $\overline{f}(g) = Res_{s=7} f^1(g,s)$.  Set
\[E_{M_{G_7,2},1,2}(g,\overline{f}) = \sum_{\gamma \in (P_{G_7,1} \cap P_{G_7,2})(\Q)\backslash P_{G_7,2}(\Q)}{\overline{f}(\gamma g)}\]
and
\[E_{M_{G_7,3},1,3}(g,\overline{f}) = \sum_{\gamma \in (P_{G_7,1} \cap P_{G_7,3})(\Q)\backslash P_{G_7,3}(\Q)}{\overline{f}(\gamma g)}.\]
These sums are absolutely convergent.

\begin{proposition} Let the notation be as above. 
\begin{enumerate}
	\item $\Theta_f(g)_{N_{P_{G_7,2}}} = E_{M_{G_7,2},1,2}(g,\overline{f})$.
	\item $\Theta_f(g)_{N_{P_{G_7,3}}} = E_{M_{G_7,3},1,3}(g,\overline{f})$.
\end{enumerate} 
\end{proposition}
\begin{proof}
One first computes the sets $[W_{M_{G_7,j}}\backslash W_{G_7} / W_{M_{G_7,1}}]$ for $j = 2,3$.  For $j = 3$, the set has size two, with elements $1,w_3$ where $w_3(r_1) = -r_3$, $w_3(r_2) = r_1$, and $w_3(r_3) = r_2$.  For $j=2$, the set has size three, with elements $1,w_1, w_2$ where
\begin{enumerate}
	\item $w_1(r_1) = r_3$, $w_1(r_2) = r_1$, $w_1(r_3) = r_2$;
	\item $w_2(r_1) = -r_2$, $w_2(r_2) = r_1$ and $w_2(r_3) = r_3$.
\end{enumerate}
Now, for $j=2$ or $3$ and $w=1$, one checks that the Eisenstein series $E^{w=1}(g,f,s)$ are regular at $s=7$ by absolute convergence.  Thus, these terms do not contribute to the residue at $s=7$.  One deduces the second part of the proposition from Langlands functional equation for Eisenstein series.

For the first part of the proposition, we analyze the Eisenstein series coming from $w = w_1$.  In this case, we are reduced to the Eisenstein series studied in Proposition \ref{prop:D5reg}.  Thus this term does not contribute to the residue.  The first part of the proposition now follows from Langlands functional equation.
\end{proof}

We now consider the constant term of $\Theta_f$ along $N_{P_{G_7,1}}$.  Let $w_{12}$ be the Weyl group element that exchanges $r_1$ with $r_2$ (leaving $r_3$ fixed).  Set $f^{w_{12}}(g,s) = M(w_{12})f(g,s)$ and $E^{w_{12}}(g,f,s)$ the associated Eisenstein series (see subsection \ref{subsec:Eis}).  

\begin{proposition} One has
	\[\Theta_f(g)_{N_{G_7,1}} = \overline{f}(g) + res_{s=7} E^{w_{12}}(g,f,s).\]
This latter Eisenstein series residue can be identified with theta functions on $G_{6,\Theta}$.
\end{proposition}
\begin{proof} The set $[W_{M_{G_7,1}}\backslash W_{G_7} / W_{M_{G_7,1}}]$ has size three, with elements $1, w_{12}$ and $w_0$.  The part of the constant term fcorresponding to $w=1$ is just $f(g,s)$, which of course is regular at $s=7$.  The $w = w_{12}$ and $w= w_0$ terms do contribute nontrivially to the residue at $s=7$.  That $res_{s=7} E^{w_{12}}(g,f,s)$ can be identified with theta functions on $G_{6,\Theta}$ follows from Proposition \ref{prop:D6minAut}.  Specifically, a priori, the residue at $s=7$ might involve terms that do not arise as theta functions, but these terms disappear by Proposition \ref{prop:D6minAut} part 3.
\end{proof}

\section{Automorphic minimal representations: type $E$}\label{sec:AMRE}
In this section, we discuss automorphic minimal representations on group of type $E$.

\subsection{Type $E_6$} Recall the group $M_J^1$ defined in section \ref{sec:groupTheory}; it is semisimple, simply connected of type $E_6$.  Moreover, it is split at every finite place. It has an $A_2$ rational root system. Specifically, there is an action of $\SL_3$ on $J$, defined by the formula $g \cdot X = gXg^{t}$.  Taking the diagonal torus of $\SL_3$ gives rise to this root system of $M_J^1$. We write $\{r_i - r_j\}_{i \neq j}$ for the roots, where $i,j \in \{1,2,3\}$.

Let $e_{11}$ be the element of $J$ with $c_1$ coordinate equal to $1$ and all other coordinates equal to $0$; see equation \eqref{eqn:XJ}.  Let $Q_{M_J^1}$ be the parabolic subgroup of $M_J^1$ that fixes the line $\Q e_{11}$.  Define $\mu: Q \rightarrow \GL_1$ as $q e_{11} = \mu(q) e_{11}$. Let $I(s) = Ind_{Q_{M_J^1}}^{M_J^1}(|\mu|^{s/2})$.  We will consider Eisenstein series associated to this induced representation.

If $T$ is the diagonal torus of $\SL_3$, as mentioned, we have a map $T \rightarrow M_J^1$.  If $t = (t_1, t_2, t_3)$, this map satisfies $\mu(t) = t_1^2$.  One has $\delta_P(t) = 8 (2r_1 - r_2 -  r_3)$, where $r_j(t) = |t_j|$, and $\rho_{P_0} = 8(r_1-r_3)$.  We have $\lambda_s = |\mu|^{s/2} \delta_{P_0}^{-1/2} = (s-8) r_1 + 8 r_3$.  (Remember that $r_1 + r_2 + r_3 = 0$ on $T$.)

We assume $f(g,s)$ is a flat section in $I(s)$, spherical at the archimedean place.  We let $E(g,f,s)$ be the associated Eisenstein series.  We are interested in this Eisenstein series at $s= 18$.  According to \cite{weissmanFJ}, the spherical vector generates the $I(s=18)$ at every finite place. 

We now have the following proposition.

\begin{proposition}\label{prop:E6reg} The Eisenstein series $E(g,f,s)$ is regular at $s=18$. \end{proposition}
\begin{proof}
Write $Q_{M_J^1} = L_{M_J^1}V_{M_J^1}$.  We compute the constant term down to Levi $L_{M_J^1}$ of $Q_{M_J^1}$.  The set $[W_{L_{M_J^1}}\backslash W / W_{L_{M_J^1}}]$ consist of $w = 1$ and $w = w_{r_{1}-r_{2}}$, the simple reflection corresponding to this root.   We have $w_{r_1-r_2}(\lambda_s) + \rho_{P_0} = 8 r_1 + (s-8) r_2$.  The global intertwining operator $M(w_{r_1-r_2})$ is absolutely convergent at $s=18$, using that $\langle \lambda_s, r_1 - r_2 \rangle = s-8$.  But now, the Eisenstein series on $L_{M_J^1}$ associated to $M(w_{r_1-r_2})f(g,s)$ at $s=18$ was proved to be regular in Proposition \ref{prop:D5reg}.  This completes the proof.
\end{proof}

\subsection{Type $E_7$}\label{subsec:minE7}
Recall from section \ref{sec:groupTheory} the group $H_J^1$; it is simply connected of type $E_7$, and split at every finite place.  In this subsection, we define and compute with the automorphic minimal representation on $H_J^1$.

There is a map $\Sp_6 \hookrightarrow H_J^1$, defined by realizing $W_J \subseteq \wedge^3 W_6 \otimes \Theta$, where $W_6$ is the standard representation of $\Sp_6$.  Let $T$ be the diagonal torus of $\Sp_6$.  Then $T$ gives $H_J^1$ a rational root system of type $C_3$.  We write $r_i \pm r_j$, $i,j \in \{1,2,3\}$ for these roots.

We let $P_{H_J^1}$ denote the Siegel parabolic subgroup of $H_J^1$, defined as the as the stabilizer of the line $\Q (0,0,0,1)$.  In terms of the $C_3$ root system, the Siegel parabolic subgroup corresponds to the simple root $2 r_3$.  We define $\lambda: P \rightarrow \GL_1$ via $p (0,0,0,1) = \lambda(p)(0,0,0,1)$. Define $j(g,Z) \in \C^\times$ via the action on $r_0(Z) := (1,-Z,Z^{\#},-N(Z))$, i.e., $g r_0(Z) = j(g,Z) r_0(g Z)$. (See \cite[Proposition 2.3.1]{pollackQDS}.) We define sections for $I(s) := Ind_{P_{H_J^1}}^{H_J^1}(|\lambda|^s)$.  Specifically, for an even integer $\ell$, let $f_{\infty,\ell}(g,s)$ be the associated flat section in $I_\infty(s)$ with $f(k,s) = j(k,i)^{\ell}$ for all $k \in K_{H_J^1,\infty}$.  Here the maximal compact subgroup $K_{H_J^1,\infty}$ is defined as $k \in H_J^1(\R)$ with $r_0(k \cdot i) = r_0(i)$.

Consider a flat section $f_{fte}(g,s) \in I_f(s)$.  Let $f_{\ell}(g,s) = f_{fte}(g,s)f_{\infty,\ell}(g,s)$.  We define an Eisenstein series $E(g,f,s) = \sum_{\gamma \in P_{H_J^1}(\Q)\backslash H_J^1(\Q)}{f(\gamma g,s)}$.  The modulus character of $P_{H_J^1}$ is $\delta_{P_{H_J^1}}(p) = |\lambda(p)|^{18}$, so the Eisenstein series converges absolutely for $Re(s) > 18$.  

We will be interested in the residue at $s=14$ when $\ell=\pm 4$.  We fix now $\ell =4$; the results for $\ell = -4$ are identical and proved identically.

We recall from \cite[Theorem 3.3]{hanzerSavin}, see also \cite{weissmanFJ}, that the $p$-adic representations $I_p(s=14), I_p(s=4)$ have a nonsplit composition series of length two.  The spherical representation is the unique irreducible subrepresentation of $I_p(s=4)$, while it is the unique irreducible quotient of $I_p(s=14)$.  Finally, as usual, the intertwining operator locally gives a defined surjection from $I_p(s=14)$ to the proper spherical subrepresentation in $I_p(s=4)$.

\begin{proposition} Let the notation be as above.
\begin{enumerate}
	\item The Eisenstein series has at most a simple pole at $s=14$.  This pole is achieved for the inducing section that is spherical at every finite place.
	\item The residue representation is nonzero and irreducible, and thus defines an intertwining map $I_f(s=14) \rightarrow \mathcal{A}(H_J^1)$.
\end{enumerate}
\end{proposition}
\begin{proof} That the Eisenstein series has at most a simple pole at $s=14$ is proved in \cite[Proposition 6.3]{hanzerSavin}.  That this pole is achieved by the vector that is spherical at every finite place is due to Kim \cite{kimMin}.
\end{proof}

We define $\Theta_f^+(g) = Res_{s=14} E(g,f_4,s)$, and $\Theta_f^{-} = Res_{s=14} E(g,f_{-4},s)$.  We compute the constant terms of $\Theta_f^+$ along the unipotent radicals of the standard maximal parabolic subgroups.  Our simple roots are $\alpha_1 = r_1 -r_2, \alpha_2 = r_2-r_3$ and $\alpha_3 = 2r_3$.  Thus, following the naming convention of subsection \ref{subsec:parabolic}, the standard maximal parabolic subgroups of $H_J^1$ are $P_{H_J^1,j} = M_{H_J^1,j} N_{H_J^1,j}$ for $j = 1,2,3$.  The Siegel parabolic occurs for $j=3$.

We begin by computing the constant term down to the Siegel Levi.  As usual, let $M(w_0)$ be the long intertwining operator.  Set $f^1(g,s) = M(w_0)f_4(g,s)$.  We will see below that $f^1(g,s)$ has at most a simple pole at $s=14$.  Let $\overline{f}(g) = Res_{s=14} f^1(g,s)$.
\begin{proposition}\label{prop:E71constant} One has $\Theta_f(g)_{N_{H_J^1,3}} = \overline{f}(g)$.
\end{proposition}
\begin{proof}	The set $[W_{M_{H_J^1,3}}\backslash W_{H_J^1}/W_{M_{H_J^1,3}}]$ has size four.  The four elements are $1$, $w_{2r_3}$, $w_{r_2 + r_3} = w_{2r_3} w_{r_2-r_3} w_{2r_3}$, $w_0 = w_{2r_3} w_{r_2-r_3} w_{r_1-r_2} w_{r_2+r_3}$, of lengths $1$, $2$, $3$, and $6$.  
	
We now compute how these Weyl elements move around our inducing character.  We begin by observing $|\lambda| = r_1 + r_2 + r_3$ and $\delta_{P_0}^{1/2}= 17 r_1 + 9 r_2 + r_3$.
\begin{enumerate}
	\item We set $\lambda_s = |\lambda|^{s} \delta_{P_0}^{-1/2} = (s-17)r_1 + (s-9) r_2 + (s-1) r_3$.
	\item applying $w_{2r_3}$, get $w_{2r_3}(\lambda_s) = (s-17)r_1 + (s-9) r_2 + (1-s) r_3$, with $\langle \lambda_s, r_3 \rangle = s-1$ and $w_{2r_3}(\lambda_s) + \rho_{P_0} = s r_1 + s r_2 + (2-s) r_3$.
	\item applying $w_{r_2-r_3}$, get $w_{r_2-r_3} w_{2r_3}(\lambda_s) = (s-17) r_1 + (1-s)r_2 + (s-9) r_3$, with $\langle w_{2r_3}(\lambda_s), r_2-r_3 \rangle = 2s-10$.
	\item applying $w_{2r_3}$, get $w_{2r_3}w_{r_2-r_3} w_{2r_3}(\lambda_s) = (s-17) r_1 + (1-s)r_2 + (9-s) r_3$, with $\langle w_{r_2-r_3}w_{2r_3}(\lambda_s), r_3\rangle = s-9$ and $w_{r_2+r_3}(\lambda_s) + \rho_{P_0} = sr_1 + (10-s) r_2 + (10-s) r_3$.  At $s=14$, this is $6(2r_1 - r_2 -r_3) + 2(r_1 + r_2 + r_3)$.
	\item applying $w_{r_1-r_2}$, get $w_{r_1-r_2} w_{r_2+r_3}(\lambda_s) = (1-s) r_1 + (s-17) r_2 + (9-s) r_3$, with $\langle w_{r_2+r_3}(\lambda_s), r_1 -r_2 \rangle = 2s-18$.
	\item applying $w_{r_2-r_3}$, get $w_{r_2-r_3}w_{r_1-r_2} w_{r_2+r_3}(\lambda_s) = (1-s) r_1 + (9-s) r_2 + (s-17) r_3$, with $\langle w_{r_1-r_2}w_{r_2+r_3}(\lambda_s), r_2-r_3 \rangle = 2s-26$.
	\item applying $w_{2r_3}$, get $w_0(\lambda_s) = (1-s) r_1 + (9-s) r_2 + (17-s) r_3$, with $\langle w_{r_2-r_3}w_{r_1-r_2}w_{r_2+r_3}(\lambda_s), r_3 \rangle = s-17$ and $w_0(\lambda_s) + \rho_{P_0} = (18-s)(r_1 + r_2 + r_3)$.
\end{enumerate}

At the finite places, we obtain the following $c$-functions.  Set $\zeta_{\Theta}(s) = \zeta(s)\zeta(s-3)$.
\begin{enumerate}
	\item $c(1) = 1$
	\item $c(w_{2r_3}) = \frac{\zeta(s-1)}{\zeta(s)}$
	\item $c(w_{r_2+r_3}) = c(w_{2r_3}) \frac{\zeta_{\Theta}(s-5)}{\zeta_{\Theta}(s-1)} \frac{\zeta(s-9)}{\zeta(s-8)} = \frac{\zeta(s-5)\zeta(s-9)}{\zeta(s)\zeta(s-4)} $
	\item $c(w_0) = c(w_{r_2+r_3}) \frac{\zeta_{\Theta}(s-9)}{\zeta_{\Theta}(s-5)} \frac{\zeta_{\Theta}(s-13)}{\zeta_{\Theta}(s-9)} \frac{\zeta(s-17)}{\zeta(s-16)} = \frac{\zeta(s-9)\zeta(s-13)\zeta(s-17)}{\zeta(s)\zeta(s-4)\zeta(s-8)}$
\end{enumerate}

Observe that the global intertwining operators $M(w_{2r_3})$ and $M(w_{r_2+r_3})$ are absolutely convergent at $s=14$.  Moreover, $w_{2r_3}(\lambda_s) + \rho_{P_0} = s r_1 + s r_2 + (2-s) r_3$ at $s=14$ becomes $(8 \frac{2}{3})(r_1 + r_2 - 2r_3) + \frac{16}{3}(r_1+r_2+r_3)$.  Because $8 \frac{2}{3} > 8$, the associated Eisenstein series is absolutely convergent at $s=14$.  Thus the $E^w(g,f,s)$ for $w=1$ and $w=w_{2r_3}$ do not contribute to the residue at $s=14$.  Additionally, the $E^w(g,f,s)$ for $w = w_{r_2+r_3}$ is regular at $s=14$, by Proposition \ref{prop:E6reg}.

The only term that can contribute is thus $M(w_0)f(g,s)$.  By explicitly computing the intertwining operator at the archimedean place, we see that $M(w_0)f(g,s)$ has at most a simple pole at $s=14$.  Indeed, because $I_p(s=14)$ is generated by the spherical vector for every $p < \infty$, it suffices to check the simplicity of the pole when the inducing section is spherical at every finite place.
\end{proof}

\begin{remark}\label{rmk:E7lowerwt} One can use identical computations to those in the proof of Proposition \ref{prop:E71constant} to prove that if $\ell \in \{-2,0,2\}$ then the Eisenstein series is regular at $s=14$.
\end{remark}

We now compute the constant term $\Theta_f(g)$ down to parabolic with Levi of type $D_{5,1} \times \SL_2$.  This is the parabolic $P_{H_J^1, 2} = M_{H_J^1,2} N_{H_J^1,2}$.  The simple roots in its Levi are $r_1 - r_2, 2r_3$.  Set 
\[E_{M_{H_J^1,2}}(g,\overline{f}) = \sum_{\gamma \in (P_{H_J^1,3} \cap P_{H_J^1,2})(\Q)\backslash P_{H_J^1,2}(\Q)}{\overline{f}(\gamma g)}.\]
The sum defining this Eisenstein series is absolutely convergent.
\begin{proposition}\label{prop:E7minConst2} One has $\Theta_f(g)_{N_{H_J^1,2}} = E_{M_{H_J^1,2}}(g,\overline{f})$.
\end{proposition}
\begin{proof} The set $[W_{M_{H_J^1,2}}\backslash W_{H_J^1}/W_{M_{H_J^1,3}}]$ has size three.  Its elements are $1, w_{r_2-r_3} w_{2r_3}$ and $w = w_{r_2-r_3}w_{r_1-r_2} w_{2r_3} w_{r_2-r_2} w_{2r_3}$

Both the global intertwining operator and the sum defining the Eisenstein series on $M_{H_J^1,2}$ are absolutely convergent for $w=1$ and $w=w_{r_2-r_3} w_{2r_3}$ at $s=14$.  Thus, these terms do not contribute to the residue at $s=14$.  The proposition follows from Langlands functional equation of Eisenstein series.
\end{proof}

The parabolic $P_{H_J^1,1} = M_{H_J^1,1} N_{H_J^1,1}$ has Levi of type $D_{6,2}$.  We now compute the constant term of $\Theta_f(g)$ along $N_{H_J^1,1}$.  Let $w_1 = w_{r_1-r_2} w_{r_2-r_3} w_{2r_3}$.  Set $f^{w_1}(g,s) = M(w_1)f(g,s)$ and $E^{w_1}(g,f,s)$ the associated Eisenstein series.
\begin{proposition} One has $\Theta_f(g)_{N_{H_J^1,1}} = Res_{s=14}E^{w_1}(g,f,s)$.
\end{proposition}
\begin{proof} The set $[W_{M_{H_J^1,1}}\backslash W_{H_J^1}/W_{M_{H_J^1,3}}]$ has size two.  Its elements are $1$ and $w_1$.  The Eisenstein series $E^w(g,f,s)$ for $w=1$ is absolutely convergent at $s=14$, so does not contribute to the residue.  The proposition follows.
\end{proof}

Note that $Res_{s=14}E^{w_1}(g,f,s)$ can be considered a theta function on $G_{6,\Theta}$ associated to $f^{w_1}$.

\subsection{Type $E_8$}\label{subsec:minE8}
The automorphic minimal representation on quaternionic $E_8$, i.e., the group $G_J$, was constructed in \cite{ganMin}.  In this subsection, we review results from \cite{ganMin} and compute the constant terms of the functions in this automorphic minimal representation along for the standard maximal parabolic subgroups of $G_J$.

We fix a maximal compact subgroup $K_{G_J,\infty} \subseteq G_J(\R)$ as in \cite[Paragraph 4.1.3]{pollackQDS}.  The Lie algebra of this compact subgroup maps to $\su_2 = V_3$, affording a three-dimensional representation of $K_{G_J,\infty}$ via the adjoint action. For a positive integer $\ell$, we have the $(2\ell+1)$-dimensional vector space $\mathbf{V}_\ell$ defined as the highest weight quotient of $S^\ell(V_3)$; this is again a representation of $K_{G_J,\infty}$.  Fixing an $\sl_2$-triple of $\su_2 \otimes \C$ gives us an associated basis $\{x^{\ell+v}y^{\ell-v}\}_{-\ell \leq v \leq \ell}$ of $\mathbf{V}_{\ell}$.

The relative root system is of type $F_4$.  The simple roots are $\alpha_1 = (0,1,-1,0)$, $\alpha_2 = (0,0,1,-1)$, $\alpha_3 = (0,0,0,1)$ and $\alpha_4 = (1/2,-1/2,-1/2,-1/2)$ in a Euclidean coordinate system.  The highest root in these coordinates is $(1,1,0,0)$.

The Heisenberg parabolic subgroup of $G_J$ is defined to be the stabilizer of the highest root space.  In terms of the decomposition $\g(J) = \sl_2 \oplus \h(J)^0 \oplus V_2 \otimes W_J$, the highest root space is spanned by $\mm{0}{1}{0}{0} \in \sl_2$.  In the notation of subsection \ref{subsec:parabolic}, it is the parabolic $P_{G_J,1} = M_{G_J,1} N_{G_J,1}$.  The derived group of the Levi $M_{G_J,1}$ is the group $H_J^1$ with a $C_3$ root system.  We define $\nu: P_{G_J,1} \rightarrow \GL_1$ as $p \mm{0}{1}{0}{0} = \nu(p) \mm{0}{1}{0}{0}$.

We will consider the induced representation $I(s) = Ind_{P_{G_J,1}}^{G_J}(|\nu|^s)$.  The modulus character $\delta_{P_{G_J}} = |\nu|^{29}$.

See \cite[Proposition 3.2]{ganMin} for the following proposition.  One can also see \cite{halawiSegal}.
\begin{proposition}\label{prop:E8degPS}The representation $I_p(s=24)$ has a unique irreducible quotient, and the representation $I_p(s=5)$ has a unique irreducible subrepresentation.  These irreducible representations are spherical.
\end{proposition}

We now have the following result from \cite{ganMin}.
\begin{proposition} For any flat inducing section $f(g,s) \in I(s)$, the Eisenstein series $E(g,f,s)$ has at most a simple pole at $s=24$.
\end{proposition}
In the course of computing constant terms of $E(g,f,s)$ to the various maximal parabolic subgroups of $G_J$, we will reprove this result (in a different way).

We write elements of the Weyl group $W_{F_4}$ in notation that indicates how they are a product of simple reflections.  Specifically, if $w_{j}$ denotes the reflection corresponding to the simple root $\alpha_j$, and $w = w_{i_1}\cdots w_{i_N}$, we denote $w$ by $[i_1,i_2,\ldots,i_N]$.

We begin by computing the constant term of $E(g,f,s)$ to the parabolic with Levi of type $D_{7,3}$.  This is the parabolic $P_{G_J,4} = M_{G_J,4} N_{G_J,4}$. It has simple roots $\alpha_1,\alpha_2, \alpha_3$ in its Levi. Let $w_3 = [4,3,2,3,4,1,2,3,2,1]$ and $f^{w_3}(g,s) = M(w_3)f(g,s)$ and $E^{w_3}(g,f,s)$ the associated Eisenstein series on $M_{G_J,4}$.

\begin{proposition} For a general flat section $f(g,s) \in I(s)$, the constant term $E_{N_{G_J,4}}(g,f,s)$ has at most a simple pole at $s=24$.  The residue is $Res_{s=24} E^{w_3}(g,f,s)$.
\end{proposition}
\begin{proof} The set $[W_{M_{G_J,4}}\backslash W/ W_{M_{G_J,1}}]$ has size three, with elements
\begin{enumerate}
	\item $[]$
	\item $[4,3,2,1]$
	\item $w_3=[4,3,2,3,4,1,2,3,2,1]$
\end{enumerate}
 We analyze the terms from $[W_{M_{G_J,4}}\backslash W/ W_{M_{G_J,1}}]$ one-by-one:
	\begin{enumerate}
		\item $[]$: The associated simple roots are $[1]$.  (See subsection \ref{subsec:Eis} for the meaning of this terminology.)  This yields an Eisenstein series associated to the $D_{6,2}$ Levi on $D_{7,3}$.  The intertwining operator is trivial, and because $s=24 > 12$, the Eisenstein series is absolutely convergent.  Thus this term is regular at $s=24$.
		\item $[4,3,2,1]$:  The associated simple roots are $[2]$.  The $2$ parabolic of $D_{7,3}$ will have Levi $\SL_2 \times D_{5,1}$.  The intertwining operator is absolutely convergent.  Setting $\lambda' = [4,3,2,1](\lambda_s) + \rho_{P_0}$, we obtain $\langle \lambda', \alpha_2^\vee \rangle = s-9$.  Because $24-9 = 15 > 10$, this Eisenstein series is absolutely convergent, so is regular at $s=24$.
		\item $w_3=[4,3,2,3,4,1,2,3,2,1]$: The associated simple roots are $[1]$.  The intertwining operator is absolutely convergent, as $s-19-3> 1$ at $s=24$.  Setting $\lambda' = [4,3,2,3,4,1,2,3,2,1](\lambda_s) + \rho_{P_0}$, we obtain $\langle \lambda', \alpha_1^\vee \rangle = s-17$.  At $s=24$, the associated Eisenstein series on $D_{7,3}$ has at most a simple pole by \cite{hanzerSavin}.
	\end{enumerate}
	This completes the proof of the proposition.
\end{proof}

We will now make a special choice of flat inducing section at the infinite place.  Namely, set $f_{\infty,\ell}(g,s) \in I(s) \otimes \mathbf{V}_{\ell}$ the flat section satisfying $f(gk,s) = k^{-1} f(g,s)$ for all $k \in K_{G_J,\infty}$ and $f(1,s) = x^{\ell} y^{\ell} \in \mathbf{V}_{\ell}$.  Up to scalar multiple, the vector $x^{\ell} y^{\ell}$ is the image in $\mathbf{V}_{\ell}$ of $h^{\ell}$, where $e,h,f \in \su_2\otimes \C$ is our fixed $\sl_2$ triple.

For a flat section $f_{fte}(g,s) \in I_f(s)$, we set $f_{\ell}(g,s) = f_{fte}(g,s) f_{\infty,\ell}(g,s)$.  \textbf{We fix $\ell=4$} and consider the Heisenberg Eisenstein series $E(g,f,s)$.

We now have the following proposition.
\begin{proposition}[Gan, see \cite{ganMin,ganSW}] The Eisenstein series $E(g,f,s)$ with $f$ spherical at every finite place attains the pole at $s=24$. The residue of the Eisenstein map $I_f(s=24) \rightarrow \mathcal{A}(G_J)$ is defined and intertwining, and the residual representation, is irreducible.
\end{proposition}
We write $\Theta_f(g) = Res_{s=24}E(g,f,s)$.

We now compute the constant terms of $\Theta_f$ along the parabolic subgroups $P_{G_J,j}$ with $j = 1,2,3$.  We begin with the constant term down to the Heisenberg parabolic.  As usual, we write $f^1(g,s) = M(w_0)f(g,s)$ and $\overline{f}(g) = Res_{s=24}f^1(g,s)$.  (We will see momentarily that $f^1(g,s)$ has at most a simple pole at $s=24$.)

\begin{proposition}  Let $w_2 = [1,2,3,4,2,3,2,1]$, $f^{w_2}(g,s) = M(w_2)f(g,s)$ and let the other notation be as above.  Then $f^1(g,s)$ has a simple pole at $s=24$ while the integral defining $M(w_2)$ is absolutely convergent.  One has 
	\[\Theta_f(g)_{N_{G_J,1}} = \overline{f}(g) + Res_{s=24}E_{w_2}(g,f,s).\]
\end{proposition}
The Eisenstein series on $M_{G_J,1}$ is for its Siegel parabolic, and yields a vector in the minimal representation on $H_J^1$.
\begin{proof} 
The set $[W_{M_{G_J,1}}\backslash W/ W_{M_{G_J,1}}]$ has five elements:
\begin{enumerate}
\item $[]$
\item $[1]$
\item $w_2=[1,2,3,4,2,3,2,1]$
\item $1,2,3,2,1]$
\item $w_0=[1,2,3,4,2,3,1,2,3,4,1,2,3,2,1]$
\end{enumerate}

We analyze the them in turn:
	\begin{enumerate}
		\item The term $[]$ yields the inducing section, which is of course defined at $s=24$.
		\item The term $[1]$ yields a Siegel Eisenstein series on the Levi, evaluated at $s=23 > 18$, (observe $|\lambda(h_{\alpha_2}(t))| = |t|$ as $ \langle r_1+r_2+r_3, r_3\rangle = 1$) and from an absolutely convergent intertwining operator.  Thus this term does not contribute to the residue.
		\item The term $[1,2,3,2,1]$ gives an Eisenstein series for $M_{G_J,1}$ with the simple root $4$ excluded.  The associated Levi in $M_{G_J,1}$ is of type $D_{5,1} \times \SL_2$.  The intertwining operator $M([1,2,3,2,1])$ is seen to be absolutely convergent. Setting $\lambda' = [1,2,3,2,1](\lambda_s) + \rho_{P_0}$, one has $\langle \lambda', \frac{1}{2}\alpha_4^\vee \rangle = s-6$.  As $24-6=18 > 8$, the associated Eisenstein series is absolutely convergent.  Thus this term does not contribute to the residue.
		\item The term $w_2=[1,2,3,4,2,3,2,1]$ yields an absolutely convergent intertwining operator.  The associated Eisenstein series on $M_{G_J,1}$ is for the Siegel parabolic. We have analyzed this Eisenstein series in subsection \ref{subsec:minE7}, see Remark \ref{rmk:E7lowerwt}.
		\item The long intertwining operator $M(w_0)$ has a simple pole at $s=24$; see Proposition 4.1.3 of \cite{pollackE8}, which handles the spherical case.  The general case follows from Proposition \ref{prop:E8degPS}.
	\end{enumerate}
\end{proof}

We next compute the constant term down to the parabolic $P_{G_J,2}=M_{G_J,2} N_{G_J,2}$.  This is the one with Levi of form $\SL_2 \times E_{6,2}$. The roots in the $\SL_2$ are $\alpha_1$, and the roots in the $E_6$ are $\alpha_3, \alpha_4$. 

\begin{proposition} Let the notation be as above.  Then $\Theta_f(g)_{N_{G_J,2}} = E_{\SL_2}(g,\overline{f})$, an absolutely convergent $\SL_2$ Eisenstein series on $M_{G_J,2}$.
\end{proposition}
\begin{proof}
 We have $[W_{M_{G_J,2}}\backslash W/ W_{M_{G_J,1}}]$:
\begin{enumerate}
	\item $[]$
	\item $[2,3,4,2,3,2,1]$
	\item $[2,3,2,1]$
	\item $[2,3,4,2,3,1,2,3,4,1,2,3,2,1]$
	\item $[2,1]$
	\item $[2,3,4,1,2,3,2,1]$
	\item $[2,3,1,2,3,4,1,2,3,2,1]$
\end{enumerate}	
	
 We handle the corresponding terms one-by-one.
	\begin{enumerate}
		\item $[]$: This yields an Eisenstein series on the $\SL_2$ part of $M_{G_J,2}$.  It is absolutely convergent, so does not contribute to the residue at $s=24$.
		\item $[2,3,4,2,3,2,1]$: This again yields an Eisenstein series on the $\SL_2$ part of $M_{G_J,2}$. The intertwining operator is absolutely convergent, and so is the Eisenstein series.  Thus this term does not contribute to the residue.
		\item $[2,3,2,1]$: The associated simple roots for this Eisenstein series are $[1,4]$.  Thus this term yields an Eisenstein series on $\SL_2$ part of the Levi, and an Eisenstein on $E_6$ part.  The intertwining operator is absolutely convergent, and so is the $\SL_2$ Eisenstein series.  Setting $\lambda' = [2,3,2,1](\lambda_s) + \rho_{P_0}$, one has $\langle \lambda_s, \frac{1}{2} \alpha_4^\vee \rangle = s-6$.  As $24-6 = 18 > 12$, this Eisenstein series on $E_6$ is also absolutely convergent.  Thus this term is regular at $s=24$.
		\item $[2,3,4,2,3,1,2,3,4,1,2,3,2,1]$: This yields an Eisenstein series on the $\SL_2$ bit. Neither the intertwining operator, nor the Eisenstein series, is absolutely convergent.  However, applying the Langlands functional equation, one obtains the Eisenstein series in the statement of the proposition.  
		\item $[2,1]$: This yields an Eisenstein series on the $E_6$ part, with simple root $[3]$ not in the new Levi. The intertwining operator is absolutely convergent, and so is the Eisenstein series.  Thus this term does not contribute to the residue.
		\item $[2,3,4,1,2,3,2,1]$: The associated roots for this term are $[1,3]$, so there is an $\SL_2$ Eisenstein series and an $E_6$ Eisenstein series.  The intertwining operator is absolutely convergent.  Setting $\lambda' = [2,3,4,1,2,3,2,1](\lambda_s) + \rho_{P_0}$, one has $\langle \lambda',\frac{1}{2} \alpha_3^\vee \rangle = s-11$.  As $24-11=13> 12$, the Eisenstein series on $E_6$ is absolutely convergent.  One also sees that the $\SL_2$ Eisenstein series is absolutely convergent.  Thus this term does not contribute to the residue.
		\item $[2,3,1,2,3,4,1,2,3,2,1]$: The associated roots for this term is $[4]$, so this term yields an Eisenstein series on $E_6$.  One sees that the intertwining operator is absolutely convergent.  Setting $\lambda' = [2,3,1,2,3,4,1,2,3,2,1](\lambda_s) + \rho_{P_0}$, one has $\langle \lambda',\frac{1}{2} \alpha_4^\vee \rangle = s-15$.  One sees, because the $K$-equivariance is preserved by the intertwining operator, that the above inducing section on $E_6$ will be spherical at the archimedean place.  Thus we know from \ref{prop:E6reg} that this Eisenstein series is regular. 
	\end{enumerate}
	The proposition is proved.
\end{proof}

The parabolic $P_{G_J,3}=M_{G_J,3}N_{G_J,3}$ has Levi of type $\SL_3 \times D_{5,1}$.  The simple roots in its Levi are $\alpha_1,\alpha_2$ (in the $\SL_3$) and $\alpha_4$ in the $D_{5,1}$. 
\begin{proposition} Let the notation be as above.  Let $E_{\SL_3}(g,\overline{f})$ be absolutely convergent Eisenstein series on $\SL_3$ for the simple root $[1]$ for the inducing section $\overline{f}$.  Then $\Theta_f(g)_{N_{G_J,3}} = E_{\SL_3}(g,\overline{f})$.
\end{proposition}
\begin{proof} 
 We have $[W_{M_{G_J,3}}\backslash W/ W_{M_{G_J,1}}]$:
\begin{enumerate}
	\item $[]$
	\item $[3,4,2,3,2,1]$
	\item $[3,2,1]$
	\item $[3,4,2,3,1,2,3,4,1,2,3,2,1]$
	\item $[3,2,3,4,1,2,3,2,1]$
\end{enumerate}	
We analyze the terms one-by-one:
	\begin{enumerate}
		\item $[]$: The associated simple roots is $[1]$.  This yields a maximal parabolic Eisenstein series on $\SL_3$, for the $(1,2)$ parabolic.  The associated Eisenstein series is absolutely convergent.
		\item $[3,4,2,3,2,1]$: The associated simple roots are $[1,2]$.  This yields a Borel Eisenstein series on $\SL_3$.  The intertwining operator is absolutely convergent.  Setting $\lambda' = [3,4,2,3,2,1](\lambda_s) + \rho_{P_0}$, one has $\langle \lambda', \alpha_1^\vee \rangle = s-10$ and $\langle \lambda', \alpha_2^\vee \rangle = s-17$.  Because $24-10 > 24-17 = 7 > 3$, this Borel Eisenstein series is absolutely convergent.  Thus, this term does not contribute to the residue.
		\item $[3,2,1]$: The associated simple roots are $[2,4]$.  This yeilds a maximal parabolic Eisenstein series on $\SL_3$ times a maximal parabolic Eisenstein series on $D_{5,1}$.  The intertwining operator is absolutely convergent.  Setting $\lambda' = [3,2,1](\lambda_s) + \rho_{P_0}$, one has $\langle \lambda', \alpha_2^\vee \rangle = s-9$ and $\langle \lambda', \frac{1}{2} \alpha_4^\vee \rangle = s-6$.  The $\SL_3$ Eisenstein series is absolutely convergent because $24-9=15>3$.  The $D_{5,1}$ Eisenstein series is absolutely convergent because $24-6 = 18 > 8$.  Thus this term does not contribute to the residue. 
		\item $[3,4,2,3,1,2,3,4,1,2,3,2,1]$: The associated simple root is $[2]$.  This yields a maximal parabolic Eisenstein series on $\SL_3$.  Neither the intertwining operator nor the Eisenstein series will be in the range of absolute convergence.  Thus we analyze it using Langlands functional equation, and obtain the Eisenstein series in the statment of the proposition.
		\item $[3,2,3,4,1,2,3,2,1]$: The associated simple roots are $[1,4]$.  This yields a maximal parabolic Eisenstein series on $\SL_3$ times a maximal parabolic Eisenstein series on $D_{5,1}$.  The intertwining operator is absolutely convergent.  Setting $\lambda' = [3,2,3,4,1,2,3,2,1](\lambda_s) + \rho_{P_0}$, one has $\langle \lambda', \alpha_1^\vee \rangle = s-17$ and $\langle \lambda', \frac{1}{2}\alpha_4^\vee \rangle = s-15$.  Because $24-17=7 >3$ and $24-15=9 > 8$, the two Eisenstein series are absolutely convergent.  Thus this term does not contribute to the residue.
	\end{enumerate}
\end{proof}

\section{Twisted Jacquet functors}\label{sec:twisedJacquet}
In this section, we compute various twisted Jacquet functors of $p$-adic minimal representations.  We will use these computations as part of the eventual proof of the Siegel-Weil theorems in section \ref{sec:Main}.

More specifically, in this section, we prove results of the following sort.  Suppose $G \times S \subseteq G'$ is a commuting pair, and $V_{min,p}$ is a minimal representation of $G'(\Q_p)$.  Let $U$ be the unipotent radical of a parabolic subgroup of $G$, and $\chi: U(\Q_p) \rightarrow \C^\times$ a non-degenerate character.  Let $(V_{min,p})_{(U,\chi)}$ be the twisted Jacquet functor.  Then in this section, we prove that the $S(\Q_p)$-coinvariants $(V_{min,p})_{(U,\chi),S(\Q_p)}$ of $(V_{min,p})_{(U,\chi)}$ are one-dimensional in various cases.

\subsection{Orbits}
To prove the one-dimensionality of the space of coinvariants as mentioned above, we will need to show that $S_E(\Q_p)$ acts transitively on the $\Q_p$ points of a certain algebraic set $\Omega_x$, in various cases.  To prove this transitivity of action, we consistently use the following method.
\begin{enumerate}
	\item We prove that $S_E(\overline{\Q}_p)$ acts transitively on $\Omega_x(\overline{\Q}_p)$;
	\item We verify that the stabilizer of a point $\lambda \in \Omega_x(\Q_p)$ is an algebraic group that is semisimple and simply-connected.
\end{enumerate}
In the above setting, it then follows that $S_E(\Q_p)$ acts transitively on $\Omega_x(\Q_p)$ using \cite[Proposition 1]{bhargavaGrossAIT} and the triviality of the Galois cohomology of a simply-connected group over a $p$-adic field.

Throughout this section, we write $C = \Theta \otimes k$ for a $p$-adic local field $k$ and $C^0$ for the subspace of trace $0$ elements.  The group $\Spin(C)$ acts on three copies of $C$.  We write a typical element $g$ of $\Spin(C)$ as $g = (g_1, g_2, g_3)$, with $g_j \in \SO(C)$.  We begin by recalling the following well-known lemma.
\begin{lemma}\label{lem:SpinCstabs}For the action of $\Spin(C)$ on $C^3$, one has the following stabilizers:
	\begin{enumerate}
		\item The set of $g \in \Spin(C)$ with $g_1(1) = 1$ is a copy of $\Spin_7$.
		\item The set of $g \in \Spin(C)$ with $g_1(1) = 1$ and $g_2(1) = 1$ is $G_2$.
		\item Suppose $v \in C^0$ has nonzero norm.  The set of $g \in \Spin(C)$ with $g_1(1) = 1$, $g_2(1) = 1$ and $g_3(v) = v$ is a copy of $\SU_3$.
	\end{enumerate}
\end{lemma}

We now recall the construction of some specific elements in $\Spin(C)$, from \cite[Section 3.6]{springerVeldkamp}. 
\begin{lemma}[\cite{springerVeldkamp}]\label{lem:trans1} For $c \in C$ an octonion with nonzero norm, let $s_c$ denote the reflection in $c$, $\ell_c$ left multiplication by $c$ and $r_c$ right multiplication by $c$.  Suppose $a_1,\ldots, a_r, b_1, \ldots, b_r \in C$ with $\prod_{i}{N(a_i)N(b_i)} = 1$.  Set $t_1 = s_{a_1} s_{b_1} \cdots s_{a_r} s_{b_r}$, $t_2 = \ell_{a_1} \ell_{b_1^*} \cdots \ell_{a_r}\ell_{b_r^*}$, and $t_3 = r_{a_1} r_{b_1^*} \cdots r_{a_r} r_{b_r^*}$.  Finally, let $\widehat{t}(x) = (t(x^*))^*$.  Then $(\widehat{t_1},t_2,t_3) \in \Spin(C)$.
\end{lemma}
\begin{proof} It is proved in \cite[section 3.6]{springerVeldkamp} prove that under the conditions above, $t_1(xy) = t_2(x)t_3(y)$ for all $x,y \in C$, and that the $t_j$ are in $\SO(C)$.  But now one checks immediately that this means  $(\widehat{t_1},t_2,t_3) \in \Spin(C)$.
\end{proof}

\begin{lemma}\label{lem:trans2} Suppose $(v_1,v_2)$ and $(v_1',v_2')$ in $C^2$ satisfies $N(v_j) = N(v_j') \neq 0$ for $j=1,2$. Then there exists $g \in \Spin(C)$ so that $g_j(v_j) = v_j'$.
\end{lemma}
\begin{proof} We first work over the algebraic closure of $k$.  By Lemma \ref{lem:trans1}, we can move $v_1$ to $v_1'$, so we can assume $v_1 = v_1' \in \overline{k} 1$. Now, there exists $u \in C^0$ with $N(u) \neq 0$ so that $(u,v_2) = 0$.  Hence $u v_2 \in C^0$.  Now we take $u' = N(v_2)^{-1/2} (uv_2)$.  Then $N(u') = N(u)$ and $(u')^{-1} (uv_2) \in \overline{k} 1$.  Because $u, u' \in C^0$, the reflections by $u,u'$ do not move $v_1 \in \overline{k} 1$.  By choosing the squareroot of $N(v_2)$ appropriately, we see that the lemma is proved over $\overline{k}$.
	
To descend from $\overline{k}$ to $k$, we use Galois cohomology, applying Lemma \ref{lem:SpinCstabs}.
\end{proof}

\begin{lemma} Suppose $E_j \simeq k \times k \times k$ for $j=1,2$ are embedded in $J$ as cubic norm structures, both inside $H_3(k) \subseteq J$.  Then there exists $m \in M_J^1$ so that $m(E_1(a,b,c)) = E_2(a,b,c)$.
\end{lemma}
\begin{proof} We can consider both $E_j$ in $M_3(k)$, and then they can be moved to one another by $\SL_3(k)$.
\end{proof}

\begin{lemma}\label{lem:transCF} We work over a $p$-adic field $k$.  Let $F$ be an \'etale quadratic extension of $k$.  Assume we have an embedding $E = k \times F \hookrightarrow J$ satisfying the assumptions in subsection \ref{subsec:SE}.  Let $C_F$ be $(F)^\perp \subseteq H_2(C)$.  Then $\Spin_{E}$ acts transitively on elements of $C_F$ with the same nonzero norm.
\end{lemma}
\begin{proof} We first work over the algebraic closure of $k$.  In that case, we can move $E = k \times F$ to $E_1 = k \times k \times k$ embedded diagonally, via some element $g \in M_J^1$.  Then the claim follows from the same claim for $E_1$, which we have already proved.
	
To descend to $k$, apply Galois cohomology and Lemma \ref{lem:SpinCstabs}.
\end{proof}

\begin{lemma}\label{lem:VEtrans}  Let $E \hookrightarrow J$ be an embedding of a cubic \'etale $k$-algebra.  Let $x \in E$ have $N(x)\neq 0$, and let $\Omega_x = \{(x,v) \in E \oplus V_E: \text{ rank one}\}$.  Then $S_{E}(k)$ acts transitively on $\Omega_x(k)$.
\end{lemma}
\begin{proof} Over an algebraic closure, we may assume $E$ is embedded diagonally in $J$.  Then an $(x,v)$ in $\Omega_x$ satisfies $x= (c_1, c_2, c_3)$ with all $c_j \neq 0$, and $v= (v_1,v_2,v_3)$ with $N(v_j) = c_{j-1}c_{j+1}$ and $v_1(v_2 v_3) = c_1 c_2 c_3$.

In this case, by Lemma \ref{lem:trans2}, we may move $v_1$ and $v_2$ to nonzero elements of $\overline{k} 1$.  Then $v_3$ is uniquely determined by the final equation in terms of $v_1, v_2$.  Thus over $\overline{k}$, there is one orbit.
	
Because $v_3$ is determined by $v_1,v_2$ under the conditions of the lemma, the stabilizer of a $v=(v_1,v_2,v_3)$ is of type $G_2$.  Thus the stabilizer is simply connected, so there is one $k$-orbit.
\end{proof}

\begin{lemma}\label{lem:WEtrans} Suppose $y = (a,b,c,d) \in W_E$ is non-degenerate.  Let $\Omega_y = \{(y,w) \in W_J=W_E \oplus V_E^2: \text{ rank one}\}$.  Then $S_E(k)$ acts transitively on $\Omega_w(k)$.
\end{lemma}
\begin{proof} Using the action of $\SL_{2,E}$ on $W_J$, we may assume $y = (1,0,c,d)$, with $d^2+4N(c) \neq 0$.  In fact, working over $\overline{k}$ for now, we may assume $d = 0$, so that $N(c) \neq 0$.
	
Now, in this case, $w = (u,v)$, with $u = (u_1,u_2,u_3) \in C^3$, $N(u_j) = -c_j$, $v$ determined by $u$, and $(u_1, u_2, u_3)_{\tr_C} = 0$.  By Lemma \ref{lem:trans2}, we can and do move $u_1, u_2$ to nonzero elements of $\overline{k} 1$.  Then $\tr(u_3) = 0$ and $N(u_3)= - c_3 \neq 0$.  But such elements are in one orbit under the action of $G_2$.  Moreover, the stabilizer is an $\SU_3$ by Lemma \ref{lem:SpinCstabs}, which is simply connected.  Thus there is one orbit over $\overline{k}$, and in fact one orbit over $k$.  This completes the proof.
\end{proof}

\subsection{Spaces of coinvariants}\label{subsec:coinvts}
For the split, simply-connected group $G_n$ over $k$ of type $D_n$, with standard representation $V_{2n} = H \oplus V_{2n-2}$, let $\Omega$ denote the nonzero isotropic vectors in $V_{2n-2}$.  Let $V_{min}$ be the minimal representation of $G_n$, which recall is the unique irreducible subrepresentation of $I(s=n-2)$, in the notation of subsection \ref{subsec:DndegPS}.

We recall the following theorem.  Let $P_{G_n} = M_{G_n} N_{G_n}$ be the maximal parabolic of $G_n$ stabilizing the line $k b_1$ in $V_{2n}$.  One can define an action of $P_{G_n}$ on  $C_c^\infty(\omega)$ as in \cite{magaardSavin}.
\begin{theorem}[Savin, Maagard-Savin]\label{thm:minlocDn} One has an exact sequence of $P_{G_n}$-modules, 
	\[0 \rightarrow C_c^\infty(\omega) \rightarrow V_{min} \rightarrow V_{min,N_{G_n}} \rightarrow 0.\]
\end{theorem}
\begin{remark} We remark that one does not need to use the exact argument of \cite{savinMin} to prove this result.  One can use the Fourier-Jacobi functor of \cite{weissmanFJ}, \cite{hanzerSavin} to obtain the theorem, if one wants.
\end{remark}

We will use the following proposition in section \ref{sec:Main}.
\begin{proposition}\label{prop:twistedJacquetDn} Let $F$ be a quadratic \'etale extension of $\Q_p$, and $E = \Q_p \times F$.  Recall that we have maps $G_{2,F} \times S_E \rightarrow G_6$ and $G_{3,F} \times S_E \rightarrow G_7$; see subsection \ref{subsec:classicalgps}.  Let $P_{2,F} \subseteq G_{2,F}$ and $P_{3,F} \subseteq G_{3,F}$ be the parabolic subgroups that stabilize the line $\Q_p b_1$ in the standard representation of these groups.
	\begin{enumerate}
		\item Let $N_{2,F}$ be the unipotent radical of $P_{2,F}$, which we identify with $F$ via the exponential map.  Suppose $x \in F$ has nonzero norm to $\Q_p$, and let $\chi_x: N_{2,F} \simeq F \rightarrow \C^\times$ be the character given by $\chi_x(y) = \psi((x,y))$.  Then the space of coinvariants $(V_{min,G_6})_{(N_{2,F},\chi_x),S_E}$ is dimension one.
		\item Let $N_{3,F}$ be the unipotent radical of $P_{3,F}$, which we identify with $H \oplus F$ via the exponential map.  Suppose $x \in H \oplus F$ is non-degenerate, and let $\chi_x: N_{3,F} \simeq H\oplus F \rightarrow \C^\times$ be the character given by $\chi_x(y) = \psi((x,y))$.  Then the space of coinvariants $(V_{min,G_7})_{(N_{3,F},\chi_x),S_E}$ is dimension one.
	\end{enumerate}
\end{proposition}
\begin{proof}We prove the first item.  The proof of the second item is identical.
	
Let 
	\[\Omega_x = \{(x,v) \in F \oplus C_F: (x,v) \in \Omega \text{ is isotropic}\}.\]
By Theorem \ref{thm:minlocDn}, the coinvariants $(V_{min,G_6})_{(N_{2,F},\chi_x)} \simeq C^\infty_c(\Omega_x)$ via the restriction map.  See \cite[Lemma 2.2]{magaardSavin} for a very similar argument.  Now the claim follows from the transitivity of the action of $S_E$ on $\Omega_x$, which is proved in Lemma \ref{lem:transCF}.
\end{proof}

We now consider similar spaces of coinvariants for the minimal representations on groups of type $E_7$ and $E_8$.  We refer the reader to \cite{ganSavin} and the references contained therein, especially section 12 of \cite{ganSavin}, for the fact that the minimal representation is the unique irreducible subrepresentation of the degenerate principal series we studied in section \ref{sec:AMRE}. 

For $E_7$, we have the following.  Let $P_{H_J^1} = M_{H_J^1} N_{H_J^1}$ be the Siegel parabolic subgroup of $H_J^1$.  Let $\Omega \subseteq J$ be the set of rank one elements.  One defines an action of $P_{H_J^1}$ on $C^\infty_c(\Omega)$ as in \cite{magaardSavin}.
\begin{theorem}[Savin, Magaard-Savin]\label{thm:minE7loc}  There is a short exact sequence of $P_{H_J^1}$ modules
\[0 \rightarrow C^\infty_c(\Omega) \rightarrow V_{min,H_J^1} \rightarrow (V_{min,H_J^1})_{N_{H_J^1}}\rightarrow 0.\]
\end{theorem}
Again, this theorem can be proved using the Fourier-Jacobi functor.

For $E_8$, let $P_{G_J} = M_{G_J} N_{G_J}$ be the Heisenberg parabolic subgroup.  Let $Z \subseteq N_{G_J}$ be the center of $N_{G_J}$, which is also highest root space of $G_J$.  Denote by $\Omega$ the rank one elements of $W_J$.  There is a representation of $P_{G_J}$ on $C^\infty_c(\Omega)$; see \cite[section 2.3]{ganSWregularized}.  The following theorem (see \cite[Section 2.3]{ganSWregularized} again) can be proved using the work in \cite[Sections 11,12]{ganSavin}.
\begin{theorem}[Gan, Savin]\label{thm:minlocE8} There is a short exact sequence of $P_{G_J}$ modules
	\[0 \rightarrow C^\infty_c(\Omega) \rightarrow (V_{min,G_J})_Z \rightarrow (V_{min,G_J})_{N_{G_J}} \rightarrow 0.\]
\end{theorem}

We can now state and prove the analogues of Proposition \ref{prop:twistedJacquetDn} that we will need in the cases of minimal representation on $E_7$ and $E_8$.

\begin{proposition} Let $E$ be a cubic \'etale algebra over $\Q_p$, and $x \in E$ an element with nonzero norm to $\Q_p$.  Recall that we have a map $\SL_{2,E} \times S_E \rightarrow H_J^1$.  Let $U_E$ be the unipotent radical of the standard Borel of $\SL_{2,E}$, which we identify with $E$ via the exponential map.  Let $\chi_x: U_E \rightarrow \C^\times$ be the character given by $\chi_x(y) = \psi((x,y))$.  Then the space of coinvariants $(V_{min,H_J^1})_{(U_E,\chi_x),S_E}$ is one-dimensional.
\end{proposition}
\begin{proof} This follows from Theorem \ref{thm:minE7loc} and Lemma \ref{lem:VEtrans}, completely similar to the proof of Proposition \ref{prop:twistedJacquetDn}.
\end{proof}

We now consider the case of minimal representation on $E_8$.
\begin{proposition} Let $E$ be a cubic \'etale algebra over $\Q_p$, and $x \in W_E$ a non-degenerate element. Recall that we have a map $G_E \times S_E \rightarrow G_J$.  Let $N_E$ be the unipotent radical of the standard Heisenberg parabolic subgroup of $G_E$.  We identify $N_E/Z$ with $W_E$ via the exponential map.  Let $\chi_x: N_E \rightarrow \C^\times$ be the character given by $\chi_x(y) = \psi((x,y))$.  Then the space of coinvariants $(V_{min,G_J})_{(N_E,\chi_x),S_E}$ is one-dimensional.
\end{proposition}
\begin{proof} This follows from Theorem \ref{thm:minlocE8} and Lemma \ref{lem:WEtrans}, completely similar to the proof of Proposition \ref{prop:twistedJacquetDn}.
\end{proof}

\section{Siegel Weil Eisenstein series I}\label{sec:SWEis}
Our Siegel-Weil theorems are identities relating a theta lift to a special value of a degenerate Eisenstein series.  We call the latter ``Siegel-Weil Eisenstein series".  In this section, we define some of these Siegel-Weil Eisenstein series and compute their constant terms along various maximal parabolic subgroups.

\subsection{The group $G_{2,F}$}
Let $F$ be a real quadratic field. The group $G_{2,F}$ acts on the $V_{2,F} = H \oplus F$. Let $P_{2,F} = M_{2,F} N_{2,F}$ be the parabolic subgroup of $G_{2,F}$ that stabilizes the line $\Q b_1$.  The action of $P_{2,F}$ on $b_1$ defines a character $\nu: P_{2,F} \rightarrow \GL_1$, and we consider the associated induced representation $I(s) = Ind_{P_{2,F}}^{G_{2,F}}(|\nu|^s)$. 

Let $v_1 = \frac{1}{\sqrt{2}}(b_1 + b_{-1})$ and let $v_2 = \frac{1}{\sqrt{2}}(1_F)$.  Observe that $(v_i,v_j) = \delta_{ij}$. Then $v_1,v_2$ can be used to define a maximal compact subgroup $K_{G_{2,F},\infty} \subseteq G_{2,F}(\R)$, and the action of $K_{G_{2,F},\infty}$ on $v_1 + i v_2 \in V_{2,F} \otimes \C$ gives a character $j(\bullet,i): K_{G_{2,F},\infty} \rightarrow \C^\times$. For an even integer $\ell$, let $f_{\infty,\ell}(g,s) \in I_\infty(s)$ be the flat section with $f_{\infty,\ell}(k,s) = j(k,i)^{\ell}$.

Let $f_{\ell}(g,s) \in I(s)$ be a global flat section, with infinite component equal to $f_{\infty,\ell}$.  Let $E(g,f_{\ell},s) = \sum_{\gamma \in P_{2,F}(\Q)\backslash G_{2,F}(\Q)}{f_{\ell}(\gamma g,s)}$ be the associated Eisenstein series. The following proposition is well-known.
\begin{proposition}\label{prop:G2Fconst} Suppose $s_0:=|\ell| > 2$.  The Eisenstein series $E(g,f_{\ell},s)$ converges absolutely at $s_0$, and the constant term $E(g,f_{\ell},s=s_0)_{N_{2,F}} = f_{\ell}(g,s = s_0)$.
\end{proposition}
\begin{proof} The proof of the proposition boils down to verifying that the archimedean intertwining operator
\begin{equation}\label{eqn:G2Farch} M(w_0)f_{\infty,\ell}(g,s=|\ell|) = \int_{N_{2,F}(\R)}{f_{\infty,\ell}(w_0 n g,s=|\ell|)\,dn} = 0.\end{equation}
As mentioned, this is well-known, and in any event, can be verified by the reader.
\end{proof}

\subsection{The group $G_{3,F}$}\label{subsec:SWEisD3}
Let $F$ be a real quadratic \'etale extension of $\Q$, i.e., either $F = \Q \times \Q$ or $F$ is a real quadratic field.  The group $G_{3,F}$ acts on the vector space $V_{6,F} = H^2 \oplus F$.  Let $P_{3,F}$ be the parabolic subgroup stabilizing the line $\Q b_1$.  The action of $P_{3,F}$ on $b_1$ defines a character $\nu: P_{3,F} \rightarrow \GL_1$, and we consider the induced representation $I(s) = Ind_{P_{3,F}}^{G_{3,F}}(|\nu|^s)$.

Let $v_j = \frac{1}{\sqrt{2}}(b_j + b_{-j})$ for $j = 1,2$ and $v_3 = \frac{1}{\sqrt{2}}(1_F)$.  Then $(v_i,v_j) = \delta_{ij}$.  From $V_3 = \mathrm{Span}(v_1,v_2,v_3)$ one obtains a maximal compact subgroup $K_{G_{3,F},\infty} \subseteq G_{3,F}(\R)$.  For an integer $\ell \geq 1$, let $f_{\infty,\ell}(g,s) \in I_\infty(s) \otimes \mathbf{V}_{\ell}$ be the flat section defined exactly as in subsection \ref{subsec:AMRD7}.  Let $f_{\ell}(g,s) \in I(s)$ be a flat section with archimedean component equal to $f_{\infty,\ell}(g,s)$.  We let $E(g,f_{\ell},s) = \sum_{\gamma \in P_{3,F}(\Q)\backslash G_{3,F}(\Q)}{f_{\ell}(\gamma g,s)}$ be the associated Eisenstein series.

We now fix $\ell = 4$ and $s_0 = 5$.  The Eisenstein series is absolutely convergent at $s=s_0=5$.

\begin{proposition} Suppose $F = \Q \times \Q$ so that $V_{6,F} = H^3$.  Let $P_{3,F;3}= M_{3,F;3}N_{3,F;3}$ be the parabolic subgroup stabilizing $\mathrm{Span}(b_1,b_2,b_3)$.  Then
	\[E(g,f_{4},s=5)_{N_{3,F;3}} = \sum_{ (P_{3,F} \cap M_{3,F;3})(\Q)\backslash M_{3,F;3}(\Q)}{ f_{4}(\gamma g,s=5)}\]
	an absolutely convergent $\SL_3$ Eisenstein series.
\end{proposition}
\begin{proof} The general form of a constant term is expressed in subsection \ref{subsec:Eis}.  There are two relevant Weyl elements: $1$ and $w$ where $w(r_1) = -r_3$, $w(r_2) = r_1$ and $w(r_3) = -r_2$.  Then $w = w_{r_2+r_3} w_{r_1-r_2}$ has length two.  One finds that for this $w$, the global intertwining operator is absolutely convergent, and the associated $\SL_3$ Eisenstein series is defined by an absolutely convergent sum.  Thus, to prove that $E^w(g,f_{4},s=5)$ is $0$, it suffices to prove that the archimedean intertwining operator $M(w)$ is $0$ on $f_{4,\infty}(g,s=5)$. One is reduced to showing the vanishing of 
\begin{equation}\label{eqn:archInt1}\int_{\R^\times \times \R^2}{|t|^{s+\ell} pr(u b_1 + v b_2 + b_{-3})^{\ell} e^{-t^2(u^2 + v^2 + 1)}\,dt\,du\,dv}\end{equation}
at $s=\ell$.  Here $pr()^{\ell}$ is the natural projection from $Sym^{\ell}(V_{6,F}\otimes \R)$ to $\mathbf{V}_{\ell}$.  One can verify the vanishing using \cite[proof of Proposition 4.1.4]{pollackNTM}.  This completes the proof.
\end{proof}

\begin{remark} There is a second standard $A_2$ maximal parabolic of $G_{3,F}$ when $F = \Q \times \Q$, defined as the stabilizer of $\mathrm{Span}(b_1,b_2,b_{-3})$.  The computation of the constant term of $E(g,f_4,s=5)$ along this parabolic is essentially identical to the computation just done. \end{remark}

\begin{proposition}  Suppose $F$ is a field.  Let $P_{3,F;2} = M_{3,F;2} N_{3,F;2}$ be the parabolic subgroup of $G_{3,F}$ that stabilizes $\mathrm{Span}(b_1,b_2)$.  Then
\[E(g,f_4,s=5)_{N_{3,F;2}} = \sum_{ (P_{3,F} \cap M_{3,F;2})(\Q)\backslash M_{3,F;2}(\Q)}{f_4(\gamma g,s=5)},\]
an absolutely convergent $\SL_2$-type Eisenstein series.
\end{proposition}
\begin{proof}
The constant term $E(g,f_4,s)_{N_{3,F;2}}$ is a sum of two Eisenstein series.  The relevant Weyl elements are $1$ and $w$, where $w = w_{r_2} w_{r_1 -r_2}$.  The term for $w=1$ gives the statement of the proposition, so we must verify that $E^w(g,f_{4},s)$ vanishes at $s=5$.

As before, the global intertwining operator $M(w)$ and the associated Eisenstein series $E^w(g, f_4,s)$ are absolutely convergent at $s=5$.  So it suffices to check that the archimedean component $M(w) f_{4,\infty}(g,s)$ vanishes at $s=5$.  To see this, one first applies $M(w_{r_1-r_2})$ to $f_{4,\infty}(g,s)$.  This intertwining operator is computed in \cite[Proposition 4.2.2]{pollackNTM} and \cite[Proposition 3.3.2]{pollackE8}.  One then applies $M(w_{r_2})$ to the result, which is the integral of equation \eqref{eqn:G2Farch}, which vanishes.
\end{proof}

We now compute the constant term long the unipotent radical $N_{3,F}$ of $P_{3,F}$.
\begin{proposition} Let $F$ be either $\Q \times \Q$ or a real quadratic field. Let $w_{12}  = w_{r_1-r_2}$ be the simple reflection corresponding to the root $r_1 - r_2$.  One has
\[E(g,f_4,s=5)_{N_{3,F}} = f_4(g) + E^{w_{12}}(g,f_4,s=5),\]
the Eisenstein series $E^{w_{12}}(g,f_4,s)$ being absolutely convergent at $s=5$.
\end{proposition}
\begin{proof}  The constant term along $N_{3,F}$ of $E(g,f_4,s)$ has three terms, $f_4(g,s)$, $E^{w_{12}}(g,f_4,s)$ and the long intertwining operator $M(w_0) f_4(g,s)$.  Everything is absolutely convergent, so to prove the proposition it suffices to verify the $M(w_0)f_{4,\infty}(g,s)$ vanishes at $s=5$.  This quickly reduces to the vanishing of the integral in equation \eqref{eqn:archInt1}.  This completes the proof.
\end{proof}

\section{Siegel Weil Eisenstein series II}\label{sec:SWEisII}
In this section, we define the Siegel-Weil Eisenstein series on $G_E$, and compute its constant terms along the various maximal parabolic subgroups.

Let $P_E = P_{G_E}$ be the Heisenberg parabolic subgroup of $G_E$, and $\nu$ the character $P_{G_E} \rightarrow \GL_1$ given by the action on the highest root space.  We consider the induced representation $I(s) = Ind_{P_{G_E}}^{G_E}(|\nu|^s)$.  

Define a flat archimedean inducing section $f_{\infty,\ell}(g,s)$ exactly as in subsection \ref{subsec:minE8}.  We will take $\ell=4$.  We consider flat sections $f_{\ell=4}(g,s) \in I(s)$ with archimedean component equal to $f_{\infty,4}(g,s)$.  Let $E(g,f_4,s)$ be the associated Eisenstein series.  It converges for $Re(s) > 5$.  We will show that the Eisenstein series is regular at $s=5$, and we will be interested in the constant term of $E(g,f_4,s=5)$ along the various maximal parabolic subgroups of $G_E$.

We break the computation into cases: $E$ is a field; $E = E_{sp} = \Q \times \Q \times\Q$, and $E = \Q \times F$ with $F$ a field.

To do many of the constant term computations below, we make precise calculations at the archimedean place. By the equivariance for the maximal compact subgroup, it always suffices to make the computation of the intertwined inducing section $M(w)f_4(g,s)$ at $g=1$.  Then, the way we do this is to factor intertwining operators into ones corresponding to simple reflections, and then to explicitly compute these latter intertwiners, using $\SL_2$ theory.  Then, what one must keep track of is how the various $\SL_2$'s sit inside the group $G_E$.

The papers \cite{pollackE8} and \cite{CDDHPRcompleted} make very similar computations in slightly different contexts.  We will use notation from these two papers, and refer the reader to \cite{pollackE8} and \cite{CDDHPRcompleted} for a more thorough explanation.

We set $A = \left(\begin{array}{ccc} 2 & 2 &1 \\ 56 & 8 & -4 \\ 140 & -20 & 6\end{array}\right)$.  This matrix is the change-of-basis matrix between $x^8 + y^8, x^6 y^2 + x^2 y^4$, $x^4y^4$ and $f_1^8 + f_2^8$, $f_1^6 f_2^2 + f_1^2 f_2^6$, $f_1^4 f_2^4$; see \cite{pollackE8}.  Let $A_1 = A^t$.  We let $d(s) = \diag(v_2(s), v_1(s), v_0(s))$, where $v_0(s) = 1$, $v_1(s) = \frac{((1-s)/2)_1}{((1+s)/2)_1}$ and $v_2(s) = \frac{((1-s)/2)_2}{((1+s)/2)_2}$.

\subsection{The case of $E$ a field}\label{subsec:GEfield}
In this case, the group $G_E$ has rational root system of type $G_2$.  The simple roots are (in a Euclidean coordinate system) $\alpha_1 = (0,1,-1)$ and $\alpha_2 = (1,-2,1)$.

One has for $[W/W_{M_1}]$ the following elements:
\begin{enumerate}
	\item $[]$
	\item $[2]$
	\item $[1,2,1,2]$
	\item $[1,2]$
	\item $[2,1,2,1,2]$
	\item $[2,1,2]$.
\end{enumerate}

For $[W_{M_1}\backslash W/ W_{M_1}]$ one has
\begin{enumerate}
	\item $[]$
	\item $[2]$
	\item $[2,1,2]$
	\item $[2,1,2,1,2]$.
\end{enumerate}

For $[W_{M_2}\backslash W / w_{M_1}]$, one has
\begin{enumerate}
	\item $[]$
	\item $[1,2,1,2]$
	\item $[1,2]$.
\end{enumerate}

One finds that (in the above Euclidean coordinates) $\rho_{P_0} = (5,-1,-4)$ and the highest root is $(2,-1,-1)$.  We thus set $\lambda_s = (2s-5,1-s,4-s)$.

The long intertwining operator $w_0=[2,1,2,1,2]$ includes all the ones of smaller length that we must study, so we write down what happens with $c$-functions for $w_0$. At each step, we compute $\langle \lambda', \frac{1}{3}\alpha_j\rangle$ if a simple reflection $[j]$ is being applied.  Note that, if $j =2 $ so that the root is long, then $\alpha_j^\vee = \frac{1}{3} \alpha_j$ is the coroot.  If $j=1$ is short, then $\alpha_j = \alpha_j^\vee$, but then the associated $c$-function is $\zeta_E(\frac{1}{3} \langle \lambda', \alpha_1^\vee \rangle)$. One has
\begin{enumerate}
	\item $\lambda_s = (2s-5,1-s,4-s)$
	\item apply $[2]$, get $\langle \lambda', \frac{1}{3}\alpha_j\rangle = s-1$, and the new $\lambda' = (s-4,s-1,5-2s)$;
	\item apply $[1]$, get $\langle \lambda', \frac{1}{3}\alpha_j\rangle = s-2$, and the new $\lambda' = (s-4,-2s+5,s-1)$;
	\item apply $[2]$, get $\langle \lambda', \frac{1}{3}\alpha_j\rangle = 2s-5$, and the new $\lambda' = (1-s,2s-5,4-s)$;
	\item apply $[1]$, get $\langle \lambda', \frac{1}{3}\alpha_j\rangle = s-3$, and the new $\lambda' = (-s+1,-s+4,2s-5)$;
	\item apply $[2]$, get $\langle \lambda', \frac{1}{3}\alpha_j\rangle = s-4$, and the new $\lambda' = (-2s+5, s-4,s-1)$.
\end{enumerate}

\begin{proposition} For the constant term $E(g,f_4(g,s))_{N_1}$ at $s=5$, set $f_4^{[2]}(g) = M([2])f_4(g,s)|_{s=5}$ (absolutely convergent) and $E_{\GL_{2,E}}(g,f_4^{[2]}(g))$, an absolutely convergent Eisenstein series on $\GL_{2,E}$.  Then $E(g,f_4,s=5)_{N_1} = f_4(g,s=5) + E_{\GL_{2,E}}(g,f_4^{[2]}(g))$.
\end{proposition}
\begin{proof} We analyze the terms in $[W_{M_1}\backslash W/ W_{M_1}]$ one-by-one:
	\begin{enumerate}
		\item $[]$: This yields the inducing section, which is regular at $s=5$, as desired.
		\item $[2]$: The intertwining operator is absolutely convergent.  Setting $\lambda'' = [2](\lambda_s) + \rho_{P_0}$, we obtain $\langle \lambda'', \frac{1}{3} \alpha_1^\vee \rangle = s-1$.  As $5-1 = 4 > 2$, this gives an absolutely convergent Eisenstein series on $\GL_{2,E}$.  
		\item $[2,1,2]$: The intertwining operator is absolutely convergent.  Setting $\lambda'' = [2,1,2](\lambda_s) + \rho_{P_0}$, we obtain $\langle \lambda'', \frac{1}{3} \alpha_1^\vee \rangle = s-2$.  As $5-2=3 > 2$, this will give us an absolutely convergent Eisenstein series on $\GL_{2,E}$. Looking at the archimedean component, we compute 
		\[v_{212}(s)=A_1^{-1}d(2s-5)A_1d(s-2)^3A_1^{-1}d(s-1)A_1(0,0,1)^t.\]
		One has $v_{212}(s=5) = (0,0,0)$, so this term disappears from the constant term $E(g,f_4,s=5)_{N_1}$.
		\item $[2,1,2,1,2]$: The intertwining operator has a global simple pole at $s=5$, and locally the integrals are absolutely convergent.  Now, we set
		\[v_{21212}(s) = A_1^{-1}d(s-4)A_1d(s-3)^3 v_{212}(s).\] 
		One finds $v_{21212}(s=5) = 0$, and $v'_{21212}(s=5) = 0$.  Consequently, this term does not contribute to the constant term of  $E(g,f_4,s=5)_{N_1}$.
	\end{enumerate}
	The proposition is proved.
\end{proof}

We now compute the constant term of $E(g,f_4,s)$ down to $M_2$.
\begin{proposition} One has $E(g,f_4,s=5)_{N_2} = E_{\GL_2}(g,f_4|_{M_2})$, an absolutely convergent Eisenstein series obtained by restricting the inducing section $f_4(g)$ to $M_2$ and evaluating at $s=5$.
\end{proposition}
\begin{proof} We analyze the terms in $[W_{M_2}\backslash W / w_{M_1}]$ one-by-one:
	\begin{enumerate}
		\item $[]$: This term gives an absolutely convergent Eisenstein series on the long root $\GL_2$.
		\item $[1,2,1,2]$: The intertwining operator is globally absolutely convergent.  Setting $\lambda'' = [1,2,1,2](\lambda_s) + \rho_{P_0}$, we have $\langle \lambda'', \alpha_2^\vee \rangle = s-3$.  This is the point where the Eisenstein series has a simple pole, with residue a one-dimensional representation.  We set
		\[v_{1212}(s) = A_1 d(s-3)^3 v_{212}(s).\]
		Then one computes that $v_{1212}(s=5) = (0,0,0)$ and $v_{1212}'(s=5) = (*,*,0)$.  Because only the trivial representation could contribute to the residue, we see that this Eisenstein vanishes at $s=5$.
		\item $[1,2]$: The intertwining operator is globally absolutely convergent.  Setting $\lambda'' = [1,2](\lambda_s)+ \rho_{P_0}$, one has $\langle \lambda'', \alpha_2^\vee \rangle = 2s-4$.  At $s=5$, we thus obtain an absolutely convergent Eisenstein series.  Setting 
		\[v_{12}(s) = A_1 d(s-2)^3 A_1^{-1} d(s-1) A_1 (0,0,1)^t\]
		we have $v_{12}(s=5) = 0$.  Thus this term does not contribute to the constant term at $s=5$.
	\end{enumerate}
\end{proof}

\subsection{The case of $E = \Q \times \Q \times \Q$}\label{subsec:SWEisD4}
In this case $G_{E}$ has a root system of type $D_4$.  In Euclidean coordinates, the simple roots are $\alpha_1 = (1,-1,0,0)$, $\alpha_2 = (0,1,-1,0)$, $\alpha_3=(0,0,1,-1)$, $\alpha_4 = (0,0,1,1)$.  The simple root corresponding to the Heisenberg parabolic is $\alpha_2$.

We again have the Eisenstein series $E(g,f_4,s)$.  We will see that it is regular at $s=5$, and we will compute the constant terms along the maximal parabolic subgroups.  The three maximal non-Heisenberg parabolic subgroups are related by triality, so we will only compute the constant term down to one of them.

The constant term down to the Heisenberg Levi involves the elements of the set $[W_{M_2}\backslash W/W_{M_2}]$, which are given as follows.  
\begin{enumerate}
	\item $[]$;
	\item $[2,4,1,2]$
	\item $[2,3,1,2]$
	\item $[2,4,3,2]$
	\item $w_0 = [2,3,1,2,4,2,3,1,2]$
	\item $[2]$
	\item $[2,4,3,1,2]$
\end{enumerate}

The constant term down to the $D_{3,3}$ Levi involves the elements of the set $[W_{M_1}\backslash W/W_{M_2}]$, which are given as follows.  We also list the simple roots corresponding to the associated new parabolic of $M_1$:
\begin{enumerate}
	\item $[]$; $[2]$
	\item $[1,2]$; $[3,4]$
	\item $[1,2,4,3,2]$; $[2]$.
\end{enumerate}

One has $\rho = (3,2,1,0)$ and the highest root is $(1,1,0,0)$.  We thus set $\lambda_s = (s-3,s-2,-1,0)$.

\begin{proposition} The constant term of $E(g,f_4,s=5)$ along $N_1$ is $E_{M_1}(g,f_4|_{M_1})$, an absolutely convergent Eisenstein series on $M_1$ associated to the simple root $[2]$ of $M_1$.
\end{proposition}
\begin{proof} We evaluate one-by-one:
	The constant term down to the $D_{3,3}$ Levi involves the elements of the set $[W_{M_1}\backslash W/W_{M_2}]$, which are given as follows.  We also list the simple roots corresponding to the associated new parabolic of $M_1$:
	\begin{enumerate}
		\item $[]$: The associated simple root is $[2]$.  This gives an absolutely convergent Eisenstein series on $D_{3,3}$ from its ``Siegel" parabolic (stabilizing an isotropic line.)  
		\item $[1,2]$: The associated simple roots are $[3,4]$.  The intertwining operator gives:
		\begin{enumerate}
			\item apply $[2]$, get $\langle \lambda', \alpha_j\rangle = s-1$, and the new $\lambda' = (s-3,-1,s-2,0)$;
			\item apply $[1]$, get $\langle \lambda', \alpha_j\rangle = s-2$, and the new $\lambda' =(-1,s-3,s-2,0)$.
		\end{enumerate}
		It is globally absolutely convergent.  Setting $\lambda'' = [1,2](\lambda_s) + \rho$, one has $\langle \lambda'', \alpha_j^\vee \rangle = s-1$ for $j = 3,4$.  Now, applying the modulus character of the $[3]$ parabolic of $M_1$ to $\alpha_3^\vee(t)$ gives $|t|^2$, and similarly applying the modulus character of the $[4]$ parabolic of $M_1$ to $\alpha_4^\vee(t)$ gives $|t|^2$.  As $5-1 = 4 >2$, this Eisenstein series is absolutely convergent.  Setting $v_{12}(s) = d(s-2) A_1^{-1} d(s-1) A_1 (0,0,1)^t$, we obtain $v_{12}(s=5) = (0,0,0)$.  Thus this term does not contribute to the constant term $E(g,f_4,s=5)_{N_1}$.
		\item $[1,2,4,3,2]$: The associated simple root is $[2]$.  The intertwining operator gives:
		\begin{enumerate}
			\item apply $[2]$, get $\langle \lambda', \alpha_j\rangle = s-1,$ and the new $\lambda' = (s-3,-1,s-2,0)$;
			\item apply $[3]$, get $\langle \lambda', \alpha_j\rangle = s-2$, and the new $\lambda' = (s-3,-1,0,s-2)$;
			\item apply $[4]$, get $\langle \lambda', \alpha_j\rangle =  s-2$, and the new $\lambda' = (s-3,-1,-s+2,0)$;
			\item apply $[2]$, get $\langle \lambda', \alpha_j\rangle = s-3$, and the new $\lambda' = (s-3,-s+2,-1,0)$;
			\item apply $[1]$, get $\langle \lambda', \alpha_j\rangle =  2s-5$, and the new $\lambda' =  (-s+2,s-3,-1,0)$.
		\end{enumerate}
		This intertwining operator is globally absolutely convergent at $s=5$. Setting $\lambda' = [1,2,4,3,2](\lambda_s) + \rho$, we have $\langle \lambda', \alpha_2^\vee \rangle = s-1$.  Applying the modulus character of the ``$2$'' parabolic of this $D_3$ to $\alpha_2^\vee(t)$ gives $|t|^4$.  Thus this Eisenstein series is at the edge of absolute convergence when $s=5$.  We set 
		\[v_{12432}(s) = d(2s-5) A_1^{-1} d(s-3) A_1 d(s-2)^2 A_1^{-1} d(s-1) A_1 (0,0,1)^t.\]
		Now $v_{12432}(s=5) = 0$ and $v'_{12432}(s=5) = (0,0,*)$.  This is an Eisenstein series associated to a group with Jordan algebra $J_2(\Q \times \Q)$, so it has a simple pole at this boundary point where $s=5$ by \cite{hanzerSavin}.  The residue is the trivial representation.  However, since the $K_\infty$ type of this Eisenstein series does not contain the trivial representation, we obtain vanishing.  Thus this term does not contribute to the constant term $E(g,f_4,s=5)_{N_1}$.
	\end{enumerate}
\end{proof}

We now compute the constant term to the Heisenberg parabolic.
\begin{proposition} The constant term $E(g,f_4,s)_{N_2}$ at $s=5$ is $f_4(g,s=5) + E_{\GL_{2,E}}(g,f_4^{[2]})$ where $f_4^{[2]}(g) = M([2])f_4(g,s=5)$.
\end{proposition}
\begin{proof} We handle the elements of $[W_{M_2}\backslash W/ W_{M_2}]$ one-by-one.
	\begin{enumerate}
		\item $[]$:  The associated simple roots are $[]$ (empty).  This gives the inducing section $f_4(g,s=5)$.
		\item $[2,4,1,2]$: The associated simple roots are $[3]$.  The intertwining operator is:
		\begin{enumerate}
			\item apply $[2]$, get $\langle \lambda', \alpha_j\rangle = s-1$, and the new $\lambda'= (s-3,-1,s-2,0)$;
			\item apply $[1]$, get $\langle \lambda', \alpha_j\rangle =  s-2$, and the new $\lambda' = (-1,s-3,s-2,0)$;
			\item apply $[4]$, get $\langle \lambda', \alpha_j\rangle =  s-2$, and the new $\lambda' = (-1,s-3,0,-s+2)$;
			\item ppply $[2]$, get $\langle \lambda', \alpha_j\rangle =  s-3$, and the new $\lambda' = (-1,0,s-3,-s+2)$.
		\end{enumerate}
		Setting $\lambda' = [2,1,4,2](\lambda_s) + \rho$, we obtain $\langle \lambda_s, \alpha_3^\vee \rangle = 2s-4$.  Thus the intertwining operator and the Eisenstein series are absolutely convergent.  One calculates the archimedean intertwiner and finds that it vanishes at $s=5$.  Thus this term does not contribute to the constant term $E(g,f_4,s=5)_{N_2}$.
		\item $[2,3,1,2]$: The associated simple roots are $[4]$.  This case is nearly identical to the previous case; there is no contribution to the constant term.
		\item $[2,4,3,2]$: The associated simple roots are $[1]$.  This case is nearly identical to the previous two cases; there is no contribution to the constant term.
		\item $w_0 = [2,3,1,2,4,2,3,1,2]$: The associated simple roots are $[]$ (empty).  One finds that the intertwining operator is locally absolutely convergent but globally has a simple pole.  One computes that the archimedean intertwining operator vanishes to order at least two at $s=5$, so this term does not contribute to the constant term along $N_2$.
		\item $[2]$: The associated simple roots are $[1,3,4]$.  This gives an intertwining operator and Eisenstein series that are both absolutely convergent, and do contribute to the constant term.
		\item $[2,4,3,1,2]$: The associated simple roots are $[1,3,4]$.  This gives an intertwining operator that is absolutely convergent globally, and Eisenstein series that is also absolutely convergent.  The archimedean intertwining operator is computed to vanish at $s=5$, so this term does not contribute.
	\end{enumerate}
\end{proof}

\subsection{The case $E = \Q \times F$}\label{subsec:SWEisB3}
In this case, $G_{E}$ has a rational root system of type $B_3$.  The simple roots (in a Euclidean coordinate system) are $\alpha_1 = (1,-1,0)$, $\alpha_2 = (0,1,-1)$, $\alpha_3 = (0,0,1)$.  The parabolic subgroups $M_j$ for $j=1,2,3$ have the following Levi types: $M_1$ has Levi of type $D_{3,3}$ with rational root system of type $B_2$; $M_2$ is the Heisenberg Levi, isogenous to $\GL_2 \times \GL_{2,F}$; $M_3$ is isogoneous to $\GL_3 \times \SO_{2,F}$.

The constant term down to the Heisenberg Levi involves terms in $[W_{M_2}\backslash W/ W_{M_2}]$, which has elements
\begin{enumerate}
	\item $[]$
	\item $[2]$
	\item $[2,3,2]$
	\item $[2,3,1,2]$
	\item $[2,3,1,2,3,1,2]$.
\end{enumerate}

The constant term down to $M_1$ involves  $[W_{M_1}\backslash W/ W_{M_2}]$, which has elements
\begin{enumerate}
	\item $[]$
	\item $[1,2]$
	\item $[1,2,3,2]$
\end{enumerate}

The constant term down to $M_3$ involves  $[W_{M_3}\backslash W/ W_{M_2}]$, which has elements
\begin{enumerate}
	\item $[]$
	\item $[3,2]$
	\item $[3,2,3,1,2]$
\end{enumerate}

One finds that $\rho_{P_0} = (3,2,1)$ and the highest root is $(1,1,0)$.  We set $\lambda_s = (s-3,s-2,-1)$.

\begin{proposition} The constant term $E(g,f_4,s=5)_{N_1} = E_{M_1}(g,f_4(g,s=5)|_{M_1})$, an absolutely convergent Eisenstein series for the $[2]$ parabolic of $M_1$.
\end{proposition}
\begin{proof} We consider the elements of $[W_{M_1}\backslash W/ W_{M_2}]$ one-by-one.
	\begin{enumerate}
		\item $[]$: The associated simple root is $[2]$.  Applying the modulus character of the $2$-parabolic of $M_1$ to $\alpha_2^\vee(t)$ gives $|t|^4$.  As $\langle (s,s,0), \alpha_2^\vee \rangle = s$, at $s=5$ this Eisenstein series is absolutely convergent.
		\item $[1,2]$: The associated simple root is $[3]$.  The intertwining operator is absolutely convergent.  Applying the modulus character $\delta_{1;3}$ of the $3$-parabolic of $M_1$ to $\alpha_3^\vee(t)$ gives $|t|^4$, so $\langle \delta_{1;3}, \frac{1}{2}\alpha_3^\vee \rangle = 2$.  Now, if $\lambda' = [1,2](\lambda_s) + \rho_{P_0}$, then $\langle \lambda', \frac{1}{2}\alpha_3^\vee \rangle = s-1$.  As $5-1=4 > 2$, this Eisenstein series is absolutely convergent.  One computes that the archimedean intertwiner vanishes at $s=5$, so this term does not contribute.
		\item $[1,2,3,2]$: The associated simple root is $[2]$.  The intertwining operator is absolutely convergent.  Setting $\lambda' = [1,2,3,2](\lambda_s) + \rho_{P_0}$, one finds $\langle \lambda', \alpha_2^\vee \rangle = s-1$.  So the Eisenstein series is at the reducibility point corresponding to the trivial representation.  It is for a group associated to the Jordan algebra $J_2(F)$, so the pole is simple by \cite{hanzerSavin}.  Set $v_{1232}(s)$ the function of $s$ from the archimedean intertwining operator.  One finds $v_{1232}(s=5) = 0$ and $v'_{1232}(s=5) = (0,0,*)$.  But the archimedean $K$-type is not trivial, so this term does not contribute.
	\end{enumerate}
\end{proof}

We now consider the constant term to the $3$-parabolic $M_3$ of $G_{\Q\times F}$. 
\begin{proposition} The constant term $E(g,f_4,s=5)_{N_3} = E_{M_3}(g,f_4(g,s=5)|_{M_3})$, an absolutely convergent Eisenstein series on $M_3$ for the parabolic associated to the simple root $[2]$.
\end{proposition} 
\begin{proof} The constant term down to $M_3$ involves $[W_{M_3}\backslash W/ W_{M_2}]$.  We consider the terms one-by-one.
	\begin{enumerate}
		\item $[]$: The associated simple root is $[2]$.  One has $\langle (s,s,0),\alpha_2^\vee \rangle = s$ and $\langle \delta_{3;2}, \alpha_2^\vee \rangle = 3$.  Thus this Eisenstein series is absolutely convergent.  
		\item $[3,2]$: The associated simple roots are $[1,2]$. The intertwining operator is globally absolutely convergent. Setting $\lambda' = [3,2](\lambda_s) + \rho_{P_0}$, one has $\langle \lambda', \alpha_1^\vee \rangle = s-1$ and $\langle \lambda', \alpha_2^\vee \rangle = s-2$. The Eisenstein series converges absolutely \cite[Lemma 4]{arthurEis}. Let $v_{32}(s)$ be the archimedean multiplier.  Then one has $v_{32}(s=5) = (0,0,0)$, so this term does not contribute.
		\item $[3,2,3,1,2]$: The associated simple root is $[1]$.  The intertwining operator is globally absolutely convergent.  Setting $\lambda' = [3,2,2,1,2](\lambda_s) + \rho_{P_0}$, one has $\langle \lambda', \alpha_1^\vee \rangle = s-2$.  Thus this Eisenstein series will be at the boundary of absolute convergence.  Let $v_{32312}(s)$ be the archimedean multiplier,
		\[v_{32312}(s) = A_1 d(s-3)^2 A_1^{-1} d(2s-5) A_1 d(s-2)^3 A_1^{-1} d(s-1) A_1 (0,0,1)^t.\]
		One has $v_{32312}(s=5) = 0$ and $v_{32312}'(s) = (0,0,*)$.  But now mirabolic Eisenstein series on $\GL_n$ have simple poles at the modulus character point, with one-dimensional residue.  As the archimedean $K$-type does not contain the trivial representation of $\SO(3) \subseteq\SL_3(\R)$, this Eisenstein series will vanish at $s=5$.
	\end{enumerate}
\end{proof}

Finally, we consider the constant term down to the Heisenberg parabolic.
\begin{proposition} One has $E(g,f_4,s=5)_{N_2} = f_4(g,s=5) + E_{\GL_{2,E}}(g,f_4^{[2]})$, where $f_4^{[2]} = M([2])(f_4(g,s=5))$.  Both the intertwining operator and the Eisenstein series are (globally) absolutely convergent.
\end{proposition}
\begin{proof} The constant term down to the Heisenberg Levi involves terms in $[W_{M_2}\backslash W/ W_{M_2}]$.  We consider the elements one-by-one.
	\begin{enumerate}
		\item $[]$: The associated simple roots are $[]$ (empty).  Here we have the inducing section, which does contribute to the constant term.
		\item $[2]$: The associated simple roots are $[1,3]$.  The intertwining operator is globally absolutely convergent.  So is the associated Eisenstein series.
		\item $[2,3,2]$: The associated simple roots are $[1]$.  The intertwining operator is globally absolutely convergent, and the Eisenstein series is as well.  One finds that the archimedean multiplier $v_{232}(s)$ vanishes at $s=5$.
		\item $[2,3,1,2]$: The associated simple roots are $[1,3]$.  The intertwining operator is globally absolutely convergent.  The Eisenstein series again is absolutely convergent.  One finds that the archimedean multiplier $v_{2312}(s)$ vanishes at $s=5$, so this term does not contribute.
		\item $[2,3,1,2,3,1,2]$: The associated simple roots are $[]$ (empty).  The intertwining operator is locally absolutely convergent, and globally has a simple pole at $s=5$.  One finds that the archimedean multiplier vanishes to order $2$ at $s=5$, so this term also does not contribute.
	\end{enumerate}
	
\end{proof}

\section{Main theorems}\label{sec:Main}
We now come to the Siegel-Weil theorems.  Throughout, we normalize Haar measure on the groups $S_E$ so that $S_E(\Q)\backslash S_E(\A)$ has measure $1$.  The proofs of the results in this section follows the strategy of \cite{ganSW}.

\subsection{The case $G_{2,F}$}
Suppose $F$ is a totally real quadratic \'etale extension of $\Q$. Recall we have a map $G_{2,F} \times S_{\Q \times F} \rightarrow G_{6}$.  Let $f(g,s) \in I_{G_6}(s)$ be a flat section with archimedean part fixed as in subsection \ref{subsec:AMRD6} and $\Theta_f \in \mathcal{A}(G_6)$ the associated element of the automorphic minimal representation.  Let $\overline{f}(g) = Res_{s=6}M(w_0)f(g,s)$.  Set
\[\Theta_f(\mathbf{1})(g) = \int_{S_E(\Q)\backslash S_E(\A)}{\Theta_f(g,h)\,dh}\]
the theta lift of the trivial representation.  Finally, let $E_{G_{2,F}}(g,\overline{f})$ be the absolutely convergent Eisenstein series on $G_{2,F}$ for the parabolic $P_{2,F}$ (stabilizing an isotropic line in $V_{4,F} = H \oplus F$.)

\begin{theorem}\label{thm:SWD2} The theta lift $\Theta_{f}(\mathbf{1})(g) = E(g, \overline{f})$\end{theorem}
\begin{proof}  First consider the case that $F$ is a field.  
	
Observe that $\Theta_f(g)_{N_{2,F}} = \Theta_f(g)_{N_{G_6,1}}$.  This follows from the fact that $\Theta_f$ only has rank $0$ and rank one Fourier coefficients along $N_{G_6,1}$.  Now, applying Proposition \ref{prop:minD6const} and Proposition \ref{prop:G2Fconst}, we observe that $\Theta_{f}(\mathbf{1})(g)$ and $E(g, \overline{f})$ have the same constant term to $P_{2,F}=M_{2,F}N_{2,F}$. Because $E(g,\overline{f})$ is orthogonal to cusp forms, we obtain that the Eisenstein part of $\Theta_{f}(\mathbf{1})(g)$ is $E(g, \overline{f})$.
	
Let $S$ be an arbitrary finite set of finite primes.  For $v \in V_{min,S}$, let $\overline{f}_{v}$ be the associated element of $I_{S,G_{2,F}}(s=4)$.  Let $\chi$ be a non-degenerate unitary character of $N_{2,F}(\A_S)$.  Set 
	\[L_\chi(v) = \int_{N_{2,F}(\A_S)}{\overline{f}_{v}(w_0 n )\chi^{-1}(n)\,dn}.\]
By Proposition \ref{prop:twistedJacquetDn}, $L_\chi$ is the unique nonzero $\chi$-linear functional on the $S_{\Q \times F}(\A_S)$ coinvariants of $V_{min,S}$, up to scalar multiple.

\begin{claim}\label{claim:phiFC0} Suppose $f \in I_{G_6,fte}(s=6)$, and $\varphi = \Theta_f(\mathbf{1})(g)$.  If the Eisenstein projection of $\varphi$ is $0$, equivalently, if $\varphi$ is cuspidal, then $\varphi = 0$.
\end{claim}
\begin{proof}
Let $\chi$ be a non-degenerate unitary character of $N_{2,F}(\Q) \backslash N_{2,F}(\A)$.  Suppose $\varphi=\Theta_f(\mathbf{1})$ is nonzero, and $\varphi_\chi(g_S g^{S}_f g_\infty) \neq 0$, for some $g = g_S g^{S}_f g_\infty \in G_{2,F}(\A)$.  Here $S$ is a finite set of finite places, and $g_S$, respectively $g^{S}_f$, denote the component of $g_f$ at the places in $S$, respectively away from $S$.  We choose $S$ large enough so that 
\begin{enumerate}
	\item $f = f_S \otimes f^S$ is a tensor
	\item $f^S$ is the normalized spherical vector
	\item $g^{S}_f \in K^S$, the product of the hyperspecial maximal compact subgroups of $G_{6}$ away from $S$.  (Recall that $G_6$ is split at every finite place.)
\end{enumerate}
With $S$ this large, we have $\varphi_\chi(g) = \varphi(g_S g_\infty) \neq 0$.

Fix the normalized spherical vectors away from $S$.  This gives an embedding $V_{min,S} \rightarrow V_{min,f}$.  Taking $\chi$-Fourier coefficients of theta lifts and evaluating at $g_\infty$ then gives a linear map $M_\chi: V_{min,S} \rightarrow \C$.  We have $M_\chi(g_S \overline{f}) \neq 0$.  The linear map factors through the $S_{\Q \times F}(\A_S)$-coinvariants of $V_{min,S}$.  Thus there is a nonzero constant $c_\chi(g_\infty)$ so that $M_\chi(v) = c_\chi(g_\infty) L_\chi(v)$ for all $v \in V_{min,S}$.
	
Now suppose that $\varphi$ is cuspidal, or equivalently, that its Eisenstein projection is $0$.  Taking the constant term of the Eisenstein series, we see that $\overline{f}|_{G_{2,F}(\A)} \equiv 0$.  But $\overline{f}$ is spherical outside $S$, and of our special form at infinity.  Consequently, the away from $S$ part of $\overline{f}$ is nonzero at the identity.  Consequently, $\overline{f}_S(x_S) = 0$ for all $x_S \in G_{2,F}(\A_S)$.  We obtain that $L_\chi(x_S\overline{f}_S)$ is identically $0$.  This contradicts the nonvanishing of $\varphi_\chi(g_S g_\infty)$.
	
We conclude that all of $\varphi$'s non-degenerate Fourier coefficients are equal to $0$, so $\varphi = 0$.
\end{proof}

Now fix $p$ to be a split place of $F$.  Fix inducing data in $I_{G_6}(s)$ away from $p$, and we let inducing data at $p$ vary.  The theta lift then gives a linear map $I_{G_6,p}(s=6) \rightarrow \mathcal{A}(G_{2,F})$.  This map is $G_{2,F}(\Q_p)$-intertwining and factors through the coinvariants $(V_{min,p})_{S_{\Q \times F}(\Q_p)}$.  Let $\tau_p$ be the image of the map.

Similarly, fixing the same data away from $p$, the absolutely convergent Eisenstein series gives a map $Eis: I_{G_{2,F},p}(s=4) \rightarrow \mathcal{A}(G_{2,F})$.  Let $\sigma_p$ denote the image of this map.  Note that, by Theorem \ref{thm:irredUnit}, $I_{G_{2,F},p}(s=4)$ is irreducible so $\sigma_p \simeq I_{G_{2,F},p}(s=4)$ or is $0$.

Because the Eisenstein projection of the theta lift is the Siegel-Weil Eisenstein series, we obtain an equivariant map $\tau_p \rightarrow \sigma_p$.  From Claim \ref{claim:phiFC0}, this map is injective.

Because $\sigma_p$ is irreducible or $0$, we obtain $\tau_p \simeq \sigma_p \simeq I_{G_{2,F},p}(s=4)$ or $\tau_p =0$.  But $ I_{G_{2,F},p}(s=4)$ is not unitarizable, by Theorem \ref{thm:irredUnit}.  Consequently, the cuspidal projection of $\tau_p$ is $0$. Consequently the Eisenstein projection on $\tau_p$ is the identity, which proves the theorem in case $F$ is a field.

The case of $F = \Q \times \Q$ goes through similarly to the case when $F$ is a field, by applying the following lemma.
\end{proof}

\begin{lemma} Let $Z$ denote the center of the Heisenberg unipotent radical $N_{G_2,2}$ on $G_6$.  Then $\Theta_{f,Z} \equiv \Theta_{f,N_{G_6,2}}$.\end{lemma}
\begin{proof} To prove this, one again only needs to use that $\Theta_f$ has only rank $0$ and rank $1$ Fourier coefficients along $N_{G_6,1}$.\end{proof}

\subsection{The case of $G_{3,F}$}
Let $F$ be a quadratic \'etale extension of $\Q$ that is totally real. Recall that we have the map $G_{3,F} \times S_{\Q \times F} \rightarrow G_7$.

Suppose $f(g,s) \in I_{G_7}(s)$ is a flat section, with our fixed vector-valued archimedean component.  Let $\Theta_f(g) = Res_{s=7} E(g,f,s)$ be the associated theta function on $G_7$.  As usual, we set $\overline{f}(g) = Res_{s=7} M(w_0)f(g,s)$. Let $E_{G_{3,F}}(g, \overline{f})$ be the absolutely convergent Siegel-Weil Eisenstein series (see subsection \ref{subsec:SWEisD3}) associated to the parabolic $P_{G_3} \subseteq G_{3,F}$ that stabilizes an isotropic line in $V_{6,F}$. 

\begin{theorem}\label{thm:SWD3} With notation as above, $\Theta_{f}(\mathbf{1})(g) = E_{G_{3,F}}(g,\overline{f})$.\end{theorem}
\begin{proof} Using the work in subsections \ref{subsec:AMRD6}, \ref{subsec:SWEisD3}, and \ref{subsec:coinvts}, the proof is nearly identical to the proof of Theorem \ref{thm:SWD2}.  We only explain the additional ingredients that are used:
	
To check that $\Theta_{f}(\mathbf{1})(g) - E_{G_{3,F}}(g,\overline{f})$ is cuspidal, besides the computations of \ref{subsec:AMRD6} and \ref{subsec:SWEisD3}, one also uses the Siegel-Weil theorem for $G_{2,F}$, i.e., Theorem \ref{thm:SWD2}.  (One has to apply the work in subsection \ref{subsec:isog} to move between isogenous groups.)

One extra point that must be checked is that a cusp form $\varphi$ in the image of the theta lift cannot be singular, i.e., if all of its non-degenerate Fourier coefficients of $\varphi$ are $0$, then $\varphi = 0$.  In fact, it is true that in this case, $\varphi$ \emph{only} has non-degenerate Fourier coefficients along the unipotent radical of the parabolic $P_{G_3,1}$.  To see this, observe that any theta function $\Theta_f$ on $G_7$ is a \emph{modular form} in the sense of \cite{pollackNTM}.  This follows from Theorem 7.0.1 of \cite{pollackNTM}.  Then the theta lift $\Theta_f(\mathbf{1})(g)$ is again a modular form on $G_{3,F}$.  (In \cite{pollackNTM}, we only worked with groups of the form $\SO(3,n)$ with $n \geq 4$, but everything carries over line-by-line for the case $n=3$.)  But then from Theorem 3.2.4 of \cite{pollackNTM}, if $\varphi$ is a cuspidal modular form, it can only have non-degenerate Fourier coefficients; this is because the generalized Whittaker functions of that theorem are unbounded for degenerate characters.
\end{proof}

\subsection{The case of $\SL_{2,E}$}
Let $E$ be a totally real cubic \'etale $\Q$-algebra. We have the map of groups $\SL_{2,E} \times S_E \rightarrow H_J^1$.  Let $P_3=P_{H_J^1,3}$ be the Siegel parabolic subgroup of $H_J^1$, and $I_{H_J^1}(s) = Ind_{P_{3}}^{H_J^1}(|\lambda|^s)$ be the induced representation studied in subsection \ref{subsec:minE7}.  

Suppose $f(g,s) \in I_{H_J^1}(s)$ is a flat section, with archimedean part fixed as in subsection \ref{subsec:minE7}.  Let $\Theta_f \in \mathcal{A}(H_J^1)$ be the associated element of the automorphic minimal representation.  For $g \in \SL_{2,E}(\A)$, we have the theta lift 
\[\Theta_f(\mathbf{1})(g) = \int_{S_E(\Q)\backslash S_E(\A)}{\Theta_f(g,h)\,dh}.\]

On the other hand, our Siegel-Weil Eisenstein series is $E_{\SL_{2,E}}(g,\overline{f})$.  Here $\overline{f}  = Res_{s=14}M(w_0)f(g,s)$ and the sum defining the Eisenstein series is over $B_{2,E}(\Q)\backslash \SL_{2,E}(\Q)$, where $B_{2,E}$ is the standard Borel subgroup.  The sum is absolutely convergent.

The Siegel-Weil theorem is:
\begin{theorem}\label{thm:SWSL2E} We have an identity $\Theta_f(\mathbf{1})(g) = E_{\SL_{2,E}}(g, \overline{f})$.
\end{theorem}
\begin{proof}
The structure of the proof is the same as for Theorems \ref{thm:SWD2} and \ref{thm:SWD3}. We only highlight the additional ingredients that are needed for Theorem \ref{thm:SWSL2E}.  

We leave the computation of the constant terms of the Siegel-Weil Eisenstein series $E_{\SL_{2,E}}(g, \overline{f})$ to the reader, as this is an $\SL_2$ calculation.  The following claims are used to compute the constant terms of $\Theta_f(\mathbf{1})(g)$.

Recall $J = E \oplus V_E$.
	\begin{claim}\label{claim:VEJ} Over our global field $\Q$, if $X \in J$ has rank at most one, and $X \in V_E$, then $X = 0$.
	\end{claim}
	\begin{proof} One has $\tr(V)^2 - \tr(V^2) = 2 \tr(V^\#)$ for all $V \in J$.  Thus if $V$ is rank at most one, $\tr(V)^2 = \tr(V^2)$.  If $V \in V_E$, then $\tr(V) = 0$, so $\tr(V^2) = 0$.  But this form is positive definite on $J$, so $V = 0$.
	\end{proof}

When $E$ is a field, Claim \ref{claim:VEJ} is enough to finish the proof of the Siegel-Weil theorem.  We now consider the case when $E = \Q \times F$ with $F$ quadratic \'etale.
\begin{claim}\label{claim:V11J} Suppose $V \in J$ is rank one, with $c_1(V) = 0$.  Then $x_2(V) = x_3(V) = 0$ as well. \end{claim}
\begin{proof} This follows from the fact that $0 = c_j(V^\#)$ for $j = 2,3$.
\end{proof}

\begin{claim}\label{claim:ZNQ} Let $Z$ denote the root space of $H_J^1$ corresponding to $e_{11} \in J$.  Let $P_{1} = M_1 N_1$ denote the $D_{6,2}$ standard parabolic of $H_J^1$. If $\Theta_f$ is a theta function on $H_J^1$, then $(\Theta_f)_Z = \Theta_{N_1}$. \end{claim}
\begin{proof} The unipotent radical $N_1 = (XY)Z$ is a Heisenberg group.  Here $YZ = N_1 \cap N_{3}$ and $X = N_1 \cap M_3$.  Claim \ref{claim:V11J} implies $(\Theta_f)_Z = (\Theta_f)_{YZ}$. Now, there is an element $J_4 \in \Sp_4 \subseteq M_1$ that exchanges $X$ with $Y$, so this proves that $(\Theta_f)_Z$ is also invariant by $X$.  The claim follows. \end{proof}

By Theorem \ref{thm:SWD2} and our computation of $\Theta_f(g)_{N_1}$ in subsection \ref{subsec:minE7}, we can now compute the constant term $\Theta_f(\mathbf{1})(g)_{Z}$ in terms of Eisenstein series on $G_{2,F}$.

Let $N_{2,F}$ be the unipotent radical of the standard Borel of $\SL_{2,F}$, thought of as sitting inside $\SL_{2,E} = \SL_2 \times \SL_{2,F}$.  We now consider the constant term of $\Theta_{f}(g)$ along $N_F$.
\begin{claim}\label{claim:Ve11} Suppose $V \in J$ is rank one, and $V$ is orthogonal to $F \hookrightarrow H_2(\Theta)$.  Then $V \in \Q e_{11}$. \end{claim}
\begin{proof} Let $U$ denote the image of $V$ under the linear projection $J \rightarrow H_2(\Theta)$.  Then if $U \neq 0$, the quadratic norm of $U$ is negative, contradicting the fact that the $c_1$ coordinate of $V^\#$ is $0$.  Thus $U=0$. Now one considers the $c_2$ and $c_3$ coordinate entry of $V^\# = 0$ to deduce $x_2(V) = x_3(V)= 0$. \end{proof}

Let $P_{2}=M_2N_2$ be the standard $D_{5,1} \times \SL_2$ parabolic subgroup of $H_J^1$.  Recall that we have $\GL_3 \subseteq \Sp_6 \subseteq H_J^1$.  Let $\gamma \in \GL_3(\Q)$ be the permutation matrix for which $Ad(\gamma)(\SL_{2,F})$ acts trivially on $e_3,f_3$ and $Ad(\gamma)(\SL_2)$ acts trivially on $e_1,e_2,f_2,f_1$.  Here $e_1,e_2,e_3,f_3,f_2,f_1$ is the standard basis of $\Sp_6$.

Conjugating the statement of Claim \ref{claim:Ve11} by $\gamma$, it can be used to prove that $\Theta_f(g)_{\gamma \cdot N_F} = \Theta_f(g)_{N_2}$. For $g \in \SL_{2,E}(\A)$ and $h \in S_E(\A)$, we then have
\[\Theta_f(g,h)_{N_F} = \Theta_f(\gamma (g,h))_{\gamma \cdot N_F} = \Theta_f((\gamma h \gamma^{-1})\gamma g)_{N_2}.\]
The constant term $\Theta_f(x)_{N_2} = E_{\SL_2}(x,\overline{f})$, an $\SL_2$-type Eisenstein series, see Proposition \ref{prop:E7minConst2}.  We obtain 
\[\Theta_f(g,h)_{N_F} = \sum_{B(\Q)\backslash \SL_2(\Q)}{\overline{f}(\mu \gamma g)}.\]
But $\overline{f}(\gamma^{-1} x) = \overline{f}(x)$, so the sum above is the $N_F$-constant term of the Siegel-Weil Eisenstein series.

The case of the Siegel-Weil theorem when $F$ is a field now follows by the argument of Theorem \ref{thm:SWD2}.  The case when $F = \Q \times \Q$ also follows, this time using the outer $S_3$ (symmetric group) action: If $\tau \in S_3 \subseteq H_J^1(\Q)$, then
\[\Theta_f(\mathbf{1})(\tau g \tau^{-1}) = \int_{[S_E]}{\Theta_f(\tau g \tau^{-1} h)\,dh} = \int_{[S_E]}{\Theta_f(\tau g  h\tau^{-1})\,dh} = \Theta_{\tau^{-1} f}(\mathbf{1})(g)\]
where we have changed variables in the integral.
\end{proof}

\subsection{The case of $G_E$}
We now state and prove the Siegel-Weil theorem for the groups $G_E \times S_E \rightarrow G_J$.  To setup the result, suppose $f(g,s) \in I_{G_J}(s)$ is a flat section, with vector-valued archimedean component fixed as in subsection \ref{subsec:minE8}.  Let $\overline{f}(g) = Res_{s=24} M(w_0)f(g,s)$.  The Siegel-Weil Eisenstein series $E_{G_E}(g,\overline{f})$ is defined in section \ref{sec:SWEisII}. Restricting $\overline{f}$ to $G_E$, one obtains an element in $I_{G_E}(s=5)$.  Extending this to a flat section $\overline{f}(g,s)$, one can define the Eisenstein series $E_{G_E}(g,\overline{f},s) = \sum_{\gamma \in P_E(\Q)\backslash G_E(\Q)}{\overline{f}(\gamma g,s)}$, where $P_E$ is the standard Heisenberg parabolic of $G_E$.  The sum converges absolutely for $Re(s) > 5$, and we proved in section \ref{sec:SWEisII} that the Eisenstein series is regular at $s=5$.

Because the Eisenstein series is regular for all flat sections with our fixed special archimedean component, the Eisenstein map gives a $G_E(\A_f)$-intertwining map $I_{G_E,fte}(s=5) \rightarrow \mathcal{A}(G_E)$.  The Siegel-Weil Eisenstein series is defined as $E_{G_E}(g,\overline{f}) := E_{G_E}(g,\overline{f},s=5)$.

The theta lift is
\[\Theta_f(\mathbf{1})(g) = \int_{S_E(\Q)\backslash S_E(\A)}{\Theta_f(g,h)\,dh}.\]

\begin{theorem}\label{thm:SWGE} With notation as above, one has an identity $\Theta_f(\mathbf{1})(g) = E_{G_E}(g,\overline{f})$.
\end{theorem}
\begin{proof}
Our first task is to compute the constant terms of $\Theta_f(\mathbf{1})(g)$ along the maximal parabolic subgroups of $G_E$, so that we may prove that $\Theta_f(\mathbf{1})(g) - E_{G_E}(g,\overline{f})$ is cuspidal.

We begin with:
\begin{claim} Let $V \in W_J$ be rank at most one, and suppose $V$ has $0$ projection to $W_E$.  Then $V=0$.\end{claim}
\begin{proof} This claim reduces immediately to Claim \ref{claim:VEJ}.\end{proof}
It follows that the constant term $\Theta_f(g)_{N_E} = \Theta_f(g)_{N_J}$ where $N_E$ is the unipotent radical of the standard Heisenberg parabolic subgroup of $G_E$ and $N_J$ is the unipotent radical of the standard Heisenberg parabolic subgroup of $G_J$.

Let $Q_E = L_E V_E$ be the long root $\GL_2$ parabolic subgroup of $G_E$, in its $G_2$-root system.  Let $Q_J = L_J V_J$ be the standard maximal parabolic subgroup of $G_J$ with simple root $\alpha_2$ in its unipotent radical, so that the Levi subgroup $L_J$ is isogenous to $\GL_2 \times M_J^1$.
\begin{claim} The constant term of $\Theta_f(g)_{V_E} = \Theta_f(g)_{V_J}$.\end{claim}
\begin{proof}
Again, one uses Claim \ref{claim:VEJ} to prove this.  Note that we use the fact that $\Theta_f$ has only rank $0$ and rank $1$ Fourier coefficients along $N_J$ and a conjugate of $N_J$.  This can be justified using the identity $\varphi_{(N_J,\chi)}(\gamma g) = \varphi_{(N_J \cdot \gamma, \chi \cdot \gamma)}(g)$ of global Fourier coefficients of an automorphic form $\varphi$ on $G_J$.  Here $\varphi_{(N_J,\chi)}$ is the $\chi$-Fourier coefficient of $\varphi$ along $N_J$.
\end{proof}
In case $E$ is a field, we can now conclude that $\Theta_f(\mathbf{1})(g) - E_{G_E}(g,\overline{f})$ is cuspidal, using our work from subsections \ref{subsec:minE8} and \ref{subsec:GEfield} and Theorem \ref{thm:SWSL2E}.

Now suppose $E = \Q \times F$ with $F$ quadratic \'etale.  Let $P_4 = M_4 N_4$ be the standard parabolic subgroup of $G_J$ with simple root $\alpha_4$ in its unipotent radical.  One can visualize this parabolic in the $F_4$ root system, using Remark \ref{rmk:F4root} below.  Let $P_{G_E,1} = M_{G_E,1} N_{G_E,1}$ be the standard parabolic of $G_E$ with the first simple root in its unipotent radical, in the ordering of simple roots given in subsection \ref{subsec:SWEisB3}.  Thus $P_{G_E,1}$ stabilizes an isotropic line in the representation $V_{6,F} = H^2 \oplus F$ and $M_{G_E,1}$ has absolute Dynkin type $D_3$. One has that $P_{G_E,1} = P_4 \cap G_E$ and $N_{G_E,1} = N_4 \cap G_E$.  

Now, one has the identity of constant terms $\Theta_f(g)_{N_{G_E,1}} = \Theta_f(g)_{N_4}$. To prove this, use Claim \ref{claim:Ve11} and the following two claims.
\begin{claim} Suppose $V \in J$ is rank one, $c_1(V) =0$, and $V^\# \in \Q e_{11}$.  Then $V \in H_2(\Theta)$. \end{claim}
\begin{claim}\label{claim:FCstabs} Suppose $\varphi$ is a quaternionic modular form on $G_J$, $\chi: N_J(\Q)\backslash N_J(\A) \rightarrow \C^\times$ is a character, and $\varphi_\chi(g)$ the corresponding Fourier coefficients.  Suppose $u \in H_J^1(\A)$ is unipotent and stabilizes $\chi$.  Then $\varphi_\chi(ug) = \varphi_\chi(g)$.
\end{claim}
\begin{proof}  We leave the proof of the first claim to the reader.  For the second claim, one uses the fact that main theorem of \cite{pollackQDS} for the formula for the generalized Whittaker function implies that $\varphi_\chi(ug) =\varphi_\chi(g)$ if $u$ is purely archimedean, and then the general case follows by an approximation argument.
\end{proof}
Note that Claim \ref{claim:FCstabs} can be applied to $\varphi = \Theta_f$, because $\Theta_f$ is a quaternionic modular form in the sense of \cite{pollackQDS}.  Indeed, for the vector that is spherical at finite places, this is proved in \cite{pollackE8}; the general cases follows because the map $f \mapsto \Theta_f$ is $G_J(\A_f)$-intertwining.  (One can also use \cite[Proposition 6.4]{ganMin} in place of Claim \ref{claim:FCstabs}.)

Using Theorem \ref{thm:SWD3} and the results of subsections \ref{subsec:minE8} and \ref{subsec:SWEisB3}, \ref{subsec:SWEisD4}, we now have that the constant term of $\Theta_f(\mathbf{1})(g) - E_{G_E}(g,\overline{f})$ along $N_{G_E,1}$ vanishes when $E = \Q \times F$.

Suppose now that $F$ is a field. Then $G_E$ has a maximal parabolic subgroup $P_{G_E,3} = M_{G_E,3} N_{G_E,3}$ with the third simple root (in the numbering of subsection \ref{subsec:SWEisB3}) in the unipotent radical.  This is the standard maximal parabolic with Levi subgroup $M_{G_E,3}$ isogenous to $\GL_3 \times \SO_{2,F}$.  The analysis of the constant term $\Theta_f(\mathbf{1})(g)_{N_{G_E,3}}$ is similar to that of $\Theta_f(\mathbf{1})(g)_{N_{G_E,1}}$.  To do the computation, it is easiest to first consider $E$ as $E = F \times \Q$ instead of $\Q \times F$, and then use a conjugation argument as in the proof of Theorem \ref{thm:SWSL2E}.

Finally, we must consider the case where $E = E_{sp} = \Q \times \Q \times \Q$.  But to handle the constant terms of $\Theta_{f}(\mathbf{1})(g)$ and $E_{G_E}(g,\overline{f})$ in this case, we can use triality to bootstrap off of the constant terms already computed above.  Specifically, if $\tau \in C_3 \subseteq S_3 \subseteq G_J(\Q)$, then
\[\Theta_f(\mathbf{1)}(\tau g \tau^{-1}) = \int_{[S_{E_{sp}}]}{\Theta_f(\tau g \tau^{-1} h)\,dh} = \int_{[S_{E_{sp}}]}{\Theta_f(\tau g  h \tau^{-1})\,dh} = \Theta_{\tau^{-1} f}(\mathbf{1})(g)\]
where we have used triality for $S_{E_{sp}}$ to make a change of variables in the integral.  One also has $E(\tau g \tau^{-1}, \overline{f}) = E(g,\overline{\tau^{-1}f})$, using that the Heisenberg parabolic $P_E(\Q)$ is stable by $\tau$ and $\overline{f}$ is left-invariant by $\tau$.

Combining the above work, we have now proved that $\Theta_f(\mathbf{1})(g) - E_{G_E}(g,\overline{f})$ is cuspidal in all cases. To finish the proof, we make a slightly different representation-theoretic argument compared to the proofs of the other Siegel-Weil theorems, because this time the induced representation $I_{G_E,p}(s=5)$ is reducible.

Fix a split place $p$ for $E$ as usual.  Note that $f \mapsto \Theta_f$ is an intertwining map, and so is the Siegel-Weil Eisenstein series, as was remarked above. Fix inducing data in $I_{G_J}(s=24)$ away from $p$. Let $\tau_p$ be the $p$-adic rerepresentation on $G_E$ coming from the theta lift, and $\sigma_p$ the $p$-adic representation coming from the Eisenstein series map $I_{G_E,p}(s=5) \rightarrow \mathcal{A}(G_{4,E})$.  

If $\chi$ is a non-degenerate unitary character of $N_E(\Q_q)$ for an arbitary finite prime $q$, one verifies that the integral
\[L_\chi(g_q \overline{f}) = \int_{N_E(\Q_q)}{\chi^{-1}(n)\overline{f}(w_0 n g_q)\,dn}\]
is absolutely convergent.  Thus, by the argument of the proof of Theorem \ref{thm:SWD2}, we have an injection $\tau_p \hookrightarrow \sigma_p$, the map given by the Eisenstein projection.  Note that, because cuspidal quaternionic modular forms only have non-degenerate Fourier coefficients (this is a consequence of the main theorem of \cite{pollackQDS}), we do not need to be concerned with singular cuspidal theta lifts $\Theta_f(\mathbf{1})(g)$.

Now, $\sigma_p$ is a quotient of $I_{G_E,p}(s=5)$, and $\tau_p$ is a subquotient of this representation.  If $\tau_p$ is $0$, there is nothing to prove, so we may assume $\tau_p \neq 0$.  The representation $I_{G_E,p}(s=5)$ has a nonsplit composition series of length two, with subrepresentation denoted $V$ and uniuqe irreducible quotient the trivial representation.  Thus either $\tau_p = V$, $\tau_p = I_{G_E,p}(s=5)$, or $\tau_p$ is the trivial representation.  In the first case, the representation $V$ is not unitarizable; see \cite[Chapter XI, section 4]{borelWallach}.  Thus the cuspidal projection of $\tau_p$ is $0$ in that case.  In the latter two cases, the cuspidal projection of $\tau_p$ must be $0$ or the trivial representation, because it is semisimple.  But by considering Fourier coefficients of cusp forms, one sees immediately that the trivial representation of $G_{E}(\Q_p)$ cannot appear in the cuspidal spectrum.  Thus in all cases, the cuspidal projection of $\tau_p$ is $0$.

Because the data at the finite places away from $p$ was arbitrary, this proves the theorem.
\end{proof}

\begin{remark}\label{rmk:F4root} The Lie algebra $\g(J)$ has a $\Z/3\Z$-grading, 
	\[\g(J) = (\sl_3 \oplus \m(J)^0) \oplus (V_3 \otimes J) \oplus (V_3 \otimes J)^\vee,\]
where $V_3$ is the standard representation of $\sl_3$.  To compare this decomposition with the $\Z/2$-grading recalled in subsection \ref{subsec:excgps}, see \cite[Paragraph 4.2.4]{pollackQDS}.  We express various elements of $\g(J)$, in the $\Z/3\Z$-grading, in terms of the $F_4$ root system.  Here $[a_1 a_2 a_3 a_4]$ denotes the root $\sum_j a_j \alpha_j$ in the $F_4$ root system. 
\begin{itemize}
	\item $E_{13} = [2342]$
	\item $E_{12} = [1000]$
	\item $ v_1 \otimes J = \left(\begin{array}{ccc} [1122] & [1121] & [1111] \\ \relax [1121] & [1120] & [1110] \\ \relax [1111] & [1110] & [1100] \end{array}\right)$
	\item $\delta_3 \otimes J^\vee = \left(\begin{array}{ccc} [1220] & [1221] & [1231] \\ \relax [1221] & [1222] & [1232] \\ \relax [1231] & [1232] & [1242] \end{array}\right)$
	\item $E_{23} = [1342]$
	\item $\m(J) = \left(\begin{array}{ccc} * & [0001] & [0011] \\ \relax -[0001] & * & [0010] \\ \relax -[0011] & -[0010] & * \end{array}\right)$
	\item $v_2 \otimes J = \left(\begin{array}{ccc} [0122] & [0121] & [0111] \\ \relax [0121] & [0120] & [0110] \\ \relax [0111] & [0110] & [0100]\end{array}\right)$
\end{itemize}
We have written $v_1,v_2,v_3$ for the standard basis of $V_3$ and $\delta_1,\delta_2,\delta_3$ for the dual basis of $V_3^\vee$.
\end{remark}

Finally, we deduce Theorem \ref{thm:introMain} as a corollary of Theorem \ref{thm:SWGE}.  Let us make precise the statement of the result.

First, we need to define functions $\Theta_f$ for general $K_{G_J,\infty}$-type vectors.  Let $V_{min,\infty}$ denote the space of the archimedean minimal representation, as defined by Gross-Wallach \cite{grossWallachI}.  Fix a basis $w_{-4},w_{-3},\ldots, w_{4}$ of the minimal $K$-type of $V_{min,\infty}$, which we identify with $\mathbf{V}_4$, and let $w_j^\vee$ be the dual basis.  Suppose $f^{\infty}(g,s) \in I_{G_J,fte}(s)$ is a flat section.  As before, let $f_{\infty,4}(g,s)$ be our specified vector-valued archimedean flat inducing section.

Now, suppose $v \in V_{min,\infty}$, and $v = u \cdot w_0$ for some $u$ in the complexified universal enveloping algebra $\mathcal{U}(\g(J)\otimes \C)$.   We set $\Theta_{f^{\infty} \otimes v}(g) = u \langle \Theta_{f}(g), w_0^\vee \rangle$.  Here $f(g,s) = f^{\infty}(g,s) f_{\infty,4}(g,s)$.  It follows from \cite[Corollary 12.12]{ganSavin} that this association is well-defined, and thus gives an intertwining map $I_{G_J,fte}(s=24) \otimes V_{min,\infty} \rightarrow \mathcal{A}(G_J)$.  We can therefore define the theta lift $\Theta_{f^\infty\otimes v}(\mathbf{1})(g) = \int_{[S_E]}{\Theta_{f^\infty \otimes v}(g,h)\,dh}$.

We now define the Siegel-Weil Eisenstein series.  Recall that $\overline{f}(g) := Res_{s=24} M(w_0) f(g,s)$.
\begin{lemma} Suppose $v \in V_{min,\infty}$, $v = u \cdot w_0$ with $u \in \mathcal{U}(\g(J) \otimes \C)$.  Then the association $v \mapsto u \cdot \langle \overline{f}(g), w_0^\vee \rangle$ gives a well-defined, $G(\A_f) \times (\g(J) \otimes \C,K_{G_J,\infty})$-equivariant map $I_f(s=24) \otimes V_{min,\infty} \rightarrow I_{G_J}(s=5)$.
\end{lemma}
\begin{proof} Consider the constant term map $\varphi \mapsto \varphi_{U_{P_0}}$ from $\mathcal{A}(G_J) \rightarrow \mathcal{A}(U_{P_0}(\A)\backslash G_J(\A))$.  Here $U_{P_0}$ is the unipotent radical of the minimal standard parabolic of $G_J$.  One sees that on the residues of Eisenstein series $Res_{s=24}E(g,f,s)$ for arbitrary flat sections $f$, there are most three terms, and they have distinct exponents.  One of these terms is the long intertwining operator $Res_{s=24} M(w_0)f(g,s)$.  Suppose $\xi$ is its exponent, restricted to $T(\R)$.  Then the $\xi$-part of the constant term recovers the long-intertwining operator for every residue $Res_{s=24}E(g,f,s)$.  Because it's a constant term, it is an intertwining map.  The lemma follows.
\end{proof}

Given $f^\infty \otimes v \in I_{G_J,f}(s=24) \otimes V_{min,\infty}$, our Siegel-Weil Eisenstein series is defined as 
\[E_{G_E}(g,u \cdot \langle \overline{f}, w_0^\vee \rangle),s=5).\]
Here we restrict $u \cdot \langle \overline{f}, w_0^\vee \rangle$ to a flat section of $I_{G_E}(s)$ and then evaluate the corresponding Eisenstein series at $s=5$.  The Eisenstein series is regular at $s=5$ because restricting $V_{min,\infty}$ to the maximal compact subgroup $K_{G_E,\infty} \subseteq G_E(\R)$, we never see the trivial representation.  (In fact, the long root $\SU_2$ never sees the trivial representation.) If we mod out by the trivial representation, the Eisenstein map 
\[\mathrm{Eis}_1:  I_{G_J,f}(s=24) \otimes V_{min,\infty} \rightarrow \mathcal{A}(G_E)/\mathbf{1}\]
becomes $G_E(\A_f) \times (\g_E,K_{G_E,\infty})$-intertwining, and factors through the $S_E(\A)$-coinvariants.

We now have:
\begin{corollary} In $\mathcal{A}(G_E)/\mathbf{1}$, we have an identity 
	\[\Theta_{f^\infty \otimes v}(\mathbf{1})(g) = E(g,u \cdot \langle\overline{f}, w_0^\vee \rangle,s=5).\]
\end{corollary}
\begin{proof} Fix the data away from $\infty$.  Then both sides are $(\g_E,K_{G_E,\infty})$-equivariant maps on the coinvariant $(V_{min,\infty})_{S_E(\R)}$ that agree on the vector $v= w_0$ by Theorem \ref{thm:SWGE}.  But by \cite{HPS}, $(V_{min,\infty})_{S_E(\R)}$ is an irreducible  $(\g_E,K_{G_E,\infty})$-module.  The corollary follows.
\end{proof}

\bibliography{nsfANT2020new}

\newcommand{\etalchar}[1]{$^{#1}$}
\providecommand{\bysame}{\leavevmode\hbox to3em{\hrulefill}\thinspace}
\providecommand{\MR}{\relax\ifhmode\unskip\space\fi MR }
\providecommand{\MRhref}[2]{%
  \href{http://www.ams.org/mathscinet-getitem?mr=#1}{#2}
}
\providecommand{\href}[2]{#2}
\begin{thebibliography}{cDD{\etalchar{+}}22}

\bibitem[Art79]{arthurEis}
James Arthur, \emph{Eisenstein series and the trace formula}, Automorphic
  forms, representations and {$L$}-functions ({P}roc. {S}ympos. {P}ure {M}ath.,
  {O}regon {S}tate {U}niv., {C}orvallis, {O}re., 1977), {P}art 1, Proc. Sympos.
  Pure Math., XXXIII, Amer. Math. Soc., Providence, RI, 1979, pp.~253--274.
  \MR{546601}

\bibitem[BG14]{bhargavaGrossAIT}
Manjul Bhargava and Benedict~H. Gross, \emph{Arithmetic invariant theory},
  Symmetry: representation theory and its applications, Progr. Math., vol. 257,
  Birkh\"{a}user/Springer, New York, 2014, pp.~33--54. \MR{3363006}

\bibitem[BW00]{borelWallach}
A.~Borel and N.~Wallach, \emph{Continuous cohomology, discrete subgroups, and
  representations of reductive groups}, second ed., Mathematical Surveys and
  Monographs, vol.~67, American Mathematical Society, Providence, RI, 2000.
  \MR{1721403}

\bibitem[Cas]{casselmanBook}
William Casselman, \emph{Introduction to the theory of admissible
  representations of $p$-adic reductive groups}.

\bibitem[cDD{\etalchar{+}}22]{CDDHPRcompleted}
Fatma \c{C}i\c{c}ek, Giuliana Davidoff, Sarah Dijols, Trajan Hammonds, Aaron
  Pollack, and Manami Roy, \emph{The completed standard {$L$}-function of
  modular forms on {$G_2$}}, Math. Z. \textbf{302} (2022), no.~1, 483--517.
  \MR{4462682}

\bibitem[Gan00a]{ganMin}
Wee~Teck Gan, \emph{An automorphic theta module for quaternionic exceptional
  groups}, Canad. J. Math. \textbf{52} (2000), no.~4, 737--756. \MR{1767400}

\bibitem[Gan00b]{ganSW}
\bysame, \emph{A {S}iegel-{W}eil formula for exceptional groups}, J. Reine
  Angew. Math. \textbf{528} (2000), 149--181. \MR{1801660}

\bibitem[Gan08]{ganSWSnitz}
\bysame, \emph{A {S}iegel-{W}eil formula for automorphic characters: cubic
  variation of a theme of {S}nitz}, J. Reine Angew. Math. \textbf{625} (2008),
  155--185. \MR{2482219}

\bibitem[Gan11]{ganSWregularized}
\bysame, \emph{A regularized {S}iegel-{W}eil formula for exceptional groups},
  Arithmetic geometry and automorphic forms, Adv. Lect. Math. (ALM), vol.~19,
  Int. Press, Somerville, MA, 2011, pp.~155--182. \MR{2906908}

\bibitem[GS05]{ganSavin}
Wee~Teck Gan and Gordan Savin, \emph{On minimal representations definitions and
  properties}, Represent. Theory \textbf{9} (2005), 46--93. \MR{2123125}

\bibitem[GS20]{ganSavinSW}
\bysame, \emph{An exceptional {S}iegel-{W}eil formula and poles of the {S}pin
  {L}-function of {${\rm PGSp}_6$}}, Compos. Math. \textbf{156} (2020), no.~6,
  1231--1261. \MR{4108871}

\bibitem[GW94]{grossWallachI}
Benedict~H. Gross and Nolan~R. Wallach, \emph{A distinguished family of unitary
  representations for the exceptional groups of real rank {$=4$}}, Lie theory
  and geometry, Progr. Math., vol. 123, Birkh\"{a}user Boston, Boston, MA,
  1994, pp.~289--304. \MR{1327538}

\bibitem[HPS96]{HPS}
Jing-Song Huang, Pavle Pand\v{z}i\'{c}, and Gordan Savin, \emph{New dual pair
  correspondences}, Duke Math. J. \textbf{82} (1996), no.~2, 447--471.
  \MR{1387237}

\bibitem[HS20]{hanzerSavin}
Marcela Hanzer and Gordan Savin, \emph{Eisenstein series arising from {J}ordan
  algebras}, Canad. J. Math. \textbf{72} (2020), no.~1, 183--201. \MR{4045970}

\bibitem[HS22]{halawiSegal}
H.~Halawi and A.~Segal, \emph{The degenerate principal series representations
  of exceptional groups of type ${E}_8$ over $p$-adic fields}, Preprint (2022).

\bibitem[Jan93]{jantzenMAMS}
Chris Jantzen, \emph{Degenerate principal series for symplectic groups}, Mem.
  Amer. Math. Soc. \textbf{102} (1993), no.~488, xiv+111. \MR{1134591}

\bibitem[Kim93]{kimMin}
Henry~H. Kim, \emph{Exceptional modular form of weight {$4$} on an exceptional
  domain contained in {${\bf C}^{27}$}}, Rev. Mat. Iberoamericana \textbf{9}
  (1993), no.~1, 139--200. \MR{1216126}

\bibitem[KR88a]{kudlaRallisI}
Stephen~S. Kudla and Stephen Rallis, \emph{On the {W}eil-{S}iegel formula}, J.
  Reine Angew. Math. \textbf{387} (1988), 1--68. \MR{946349}

\bibitem[KR88b]{kudlaRallisII}
\bysame, \emph{On the {W}eil-{S}iegel formula. {II}. {T}he isotropic convergent
  case}, J. Reine Angew. Math. \textbf{391} (1988), 65--84. \MR{961164}

\bibitem[KR94]{kudlaRallisAnnals}
\bysame, \emph{A regularized {S}iegel-{W}eil formula: the first term identity},
  Ann. of Math. (2) \textbf{140} (1994), no.~1, 1--80. \MR{1289491}

\bibitem[Mil17]{milneBook}
J.~S. Milne, \emph{Algebraic groups}, Cambridge Studies in Advanced
  Mathematics, vol. 170, Cambridge University Press, Cambridge, 2017, The
  theory of group schemes of finite type over a field. \MR{3729270}

\bibitem[MS97]{magaardSavin}
K.~Magaard and G.~Savin, \emph{Exceptional {$\Theta$}-correspondences. {I}},
  Compositio Math. \textbf{107} (1997), no.~1, 89--123. \MR{1457344}

\bibitem[Mui97]{muicG2}
Goran Mui\'{c}, \emph{The unitary dual of {$p$}-adic {$G_2$}}, Duke Math. J.
  \textbf{90} (1997), no.~3, 465--493. \MR{1480543}

\bibitem[Pol18]{pollackLL}
Aaron Pollack, \emph{Lifting laws and arithmetic invariant theory}, Camb. J.
  Math. \textbf{6} (2018), no.~4, 347--449. \MR{3870360}

\bibitem[Pol20a]{pollackQDS}
\bysame, \emph{The {F}ourier expansion of modular forms on quaternionic
  exceptional groups}, Duke Math. J. \textbf{169} (2020), no.~7, 1209--1280.
  \MR{4094735}

\bibitem[Pol20b]{pollackE8}
\bysame, \emph{The minimal modular form on quaternionic ${E}_8$}, Jour. Inst.
  Math. Juss. (accepted) (2020).

\bibitem[Pol22]{pollackNTM}
\bysame, \emph{Modular forms on indefinite orthogonal groups of rank three}, J.
  Number Theory \textbf{238} (2022), 611--675, With appendix ``Next to minimal
  representation'' by Gordan Savin. \MR{4430112}

\bibitem[Pol23]{pollackETF}
\bysame, \emph{Exceptional theta functions and arithmeticity of modular forms
  on ${G}_2$}, Preprint (2023).

\bibitem[Ral87]{rallisBookOscillator}
Stephen Rallis, \emph{{$L$}-functions and the oscillator representation},
  Lecture Notes in Mathematics, vol. 1245, Springer-Verlag, Berlin, 1987.
  \MR{887329}

\bibitem[Sav94]{savinMin}
Gordan Savin, \emph{Dual pair {$G_{ J}\times{\rm PGL}_2$} [where] {$G_{J}$} is
  the automorphism group of the {J}ordan algebra {${J}$}}, Invent. Math.
  \textbf{118} (1994), no.~1, 141--160. \MR{1288471}

\bibitem[Sie51]{siegelIndefiniteI}
Carl~Ludwig Siegel, \emph{Indefinite quadratische {F}ormen und
  {F}unktionentheorie. {I}}, Math. Ann. \textbf{124} (1951), 17--54. \MR{67930}

\bibitem[SV00]{springerVeldkamp}
Tonny~A. Springer and Ferdinand~D. Veldkamp, \emph{Octonions, {J}ordan algebras
  and exceptional groups}, Springer Monographs in Mathematics, Springer-Verlag,
  Berlin, 2000. \MR{1763974}

\bibitem[Wei65]{weilSW}
Andr\'{e} Weil, \emph{Sur la formule de {S}iegel dans la th\'{e}orie des
  groupes classiques}, Acta Math. \textbf{113} (1965), 1--87. \MR{223373}

\bibitem[Wei03]{weissmanFJ}
Martin~H. Weissman, \emph{The {F}ourier-{J}acobi map and small
  representations}, Represent. Theory \textbf{7} (2003), 275--299. \MR{1993361}

\bibitem[Wei06]{weissmanD4}
\bysame, \emph{{$D_4$} modular forms}, Amer. J. Math. \textbf{128} (2006),
  no.~4, 849--898. \MR{2251588}

\end{thebibliography}
\bibliographystyle{amsalpha}
\end{document}